\documentclass[12pt]{amsart}
\usepackage{amscd,amsmath,amsthm,amssymb}
\usepackage[left]{lineno}
\usepackage{pst-plot,pst-3d}
\usepackage[T1]{fontenc}
\usepackage{color}
\usepackage{pstricks}
\usepackage{stmaryrd}
\usepackage[utf8]{inputenc}
\usepackage{pstricks-add}
\usepackage{graphicx}
\usepackage{epstopdf}
\usepackage{graphicx}
\usepackage{cleveref}
\usepackage{leftidx}
\usepackage{tikz}
\usepackage{tikz-cd}

\definecolor{verylight}{gray}{0.97}
\definecolor{light}{gray}{0.9}
\definecolor{medium}{gray}{0.85}
\definecolor{dark}{gray}{0.6}

 \def\NZQ{\mathbb}               

 \def\ZZ{{\NZQ Z}}

 \def\DD{{\NZQ D}}
 \def\FF{{\NZQ F}}
 \def\GG{{\NZQ G}}
 \def\HH{{\NZQ H}}
 \def\EE{{\NZQ E}}

 \def\frk{\mathfrak}               

 \def\mm{{\frk m}}
 
 \def\nn{{\frk n}}
 \def\Jc{{\mathcal J}}
 
 \def\G{{\mathcal G}}

 \def\B{{\mathcal B}}
 \def\P{{\mathcal P}}
 \def\I{{\mathcal I}}
 \def\J{{\mathcal J}}
 \def\U{{\mathcal U}}
 \def\X{{\mathcal X}}

 \def\A{{\mathcal A}}
 \def\Sc{{\mathcal S}}

 \def\ab{{\mathbf a}}
 \def\bb{{\mathbf b}}
 \def\xb{{\mathbf x}}

 \def\cb{{\mathbf c}}

 \def\0b{{\mathbf 0}}

 \def\opn#1#2{\def#1{\operatorname{#2}}} 
 \opn\chara{char} \opn\length{\ell} \opn\pd{pd} \opn\rk{rk}
 \opn\projdim{proj\,dim} \opn\injdim{inj\,dim} \opn\rank{rank}
 \opn\depth{depth} \opn\grade{grade} \opn\height{height}
 \opn\embdim{emb\,dim} \opn\codim{codim}
 
 \opn\Tr{Tr} \opn\bigrank{big\,rank}
 \opn\superheight{superheight}\opn\lcm{lcm}
 \opn\trdeg{tr\,deg}
 \opn\reg{reg} \opn\lreg{lreg} \opn\ini{in} \opn\lpd{lpd}
 \opn\size{size} \opn\sdepth{sdepth}
 \opn\link{link}\opn\fdepth{fdepth}\opn\lex{lex}
 \opn\tr{tr}
 \opn\type{type}
 \opn\gap{gap}
 \opn\arithdeg{arith-deg}
 \opn\HS{HS}
 \opn\tet{tet}
 \opn\div{div} \opn\Div{Div} \opn\cl{cl} \opn\Cl{Cl}
 \opn\Spec{Spec} \opn\Supp{Supp} \opn\supp{supp} \opn\Sing{Sing}
 \opn\Ass{Ass} \opn\Min{Min}\opn\Mon{Mon}
 \opn\Ann{Ann} \opn\Rad{Rad} \opn\Soc{Soc}\opn\Deg{Deg} \opn\Gen{Gen}
 \opn\Im{Im} \opn\Ker{Ker} \opn\Coker{Coker} \opn\Am{Am}
 \opn\Hom{Hom} \opn\Tor{Tor} \opn\Ext{Ext} \opn\End{End}
 \opn\Aut{Aut} \opn\id{id}
 
 \opn\nat{nat}
 \opn\pff{pf}
 \opn\Pf{Pf} \opn\GL{GL} \opn\SL{SL} \opn\mod{mod} \opn\ord{ord}
 \opn\Gin{Gin} \opn\Hilb{Hilb}\opn\sort{sort}
 \opn\PF{PF}\opn\Ap{Ap}
 \opn\mult{mult}
 \opn\bight{bight}
 \opn\aff{aff}
 \opn\relint{relint} \opn\st{st}
 \opn\lk{lk} \opn\cn{cn} \opn\core{core} \opn\vol{vol}  \opn\inp{inp} \opn\nilpot{nilpot}
 \opn\link{link} \opn\star{star}\opn\lex{lex}\opn\set{set}
 \opn\width{wd}
 \opn\Fr{F}
 \opn\QF{QF}
 \opn\G{G}
 \opn\type{type}\opn\res{res}
 \opn\conv{conv}
 \opn\Ind{Ind}
 \opn\soc{soc}
 \opn\gr{gr}
 
 \def\pot#1#2{#1[\kern-0.28ex[#2]\kern-0.28ex]}

 \opn\dirlim{\underrightarrow{\lim}}
 \opn\inivlim{\underleftarrow{\lim}}
 \let\dirsum=\oplus
 \let\tensor=\otimes
 \let\iso=\cong

 \let\to=\rightarrow
 \let\To=\longrightarrow
 \def\Implies{\ifmmode\Longrightarrow \else
         \unskip${}\Longrightarrow{}$\ignorespaces\fi}
 \def\implies{\ifmmode\Rightarrow \else
         \unskip${}\Rightarrow{}$\ignorespaces\fi}
 \def\iff{\ifmmode\Longleftrightarrow \else
         \unskip${}\Longleftrightarrow{}$\ignorespaces\fi}

 \let\:=\colon
 \newtheorem{Theorem}{Theorem}[section]
 \newtheorem{Lemma}[Theorem]{Lemma}
 \newtheorem{Corollary}[Theorem]{Corollary}
 \newtheorem{Proposition}[Theorem]{Proposition}
 \theoremstyle{definition}
 \newtheorem{Remark}[Theorem]{Remark}
 
 \newtheorem{Example}[Theorem]{Example}

 \let\epsilon\varepsilon
 \let\kappa=\varkappa
 \opn\dis{dis}
 \def\pnt{{\raise0.5mm\hbox{\large\bf.}}}
 
 \opn\Lex{Lex}

\newcommand{\precdot}{\prec\mathrel{\mkern-3mu}\mathrel{\cdot}}
 
\begin{document}

\title{Rings of Teter type}
\author{ Oleksandra Gasanova, J\"urgen Herzog, Takayuki Hibi, Somayeh Moradi }
\date{\today }

\address{Oleksandra Gasanova, Department of Mathematics, Uppsala University, SE-75106 Uppsala, Sweden}
\email{oleksandra.gasanova@math.uu.se}

\address{J\"urgen Herzog, Fakult\"at f\"ur Mathematik, Universit\"at Duisburg-Essen, 45117
Essen, Germany} \email{juergen.herzog@uni-essen.de}

\address{Takayuki Hibi, Department of Pure and Applied Mathematics,
Graduate School of Information Science and Technology, Osaka
University, Suita, Osaka 565-0871, Japan}
\email{hibi@math.sci.osaka-u.ac.jp}

\address{Somayeh Moradi, Department of Mathematics, School of Science, Ilam University,
P.O.Box 69315-516, Ilam, Iran}
\email{so.moradi@ilam.ac.ir}

\thanks{This paper was written during the visit of the first, the third and the fourth author to the department of Mathematics, University of Duisburg-Essen, Germany. They would like to thank the second author Professor J\"urgen Herzog for his support and hospitality.
The first author's stay was supported by Liljewalchs and Thelins travel scholarships from Uppsala University.
The third author was partially supported by JSPS KAKENHI 19H00637.
The fourth author was supported by the CIMPA's research in pairs fellowship and a partial grant from Ilam University.}

\subjclass[2010]{Primary 13H10; Secondary 13M05.}
\keywords{Teter rings, canonical trace, rings of Teter type, $0$-dimensional monomial algebras}

\maketitle

\begin{abstract}
Let $R$ be a Cohen--Macaulay local $K$-algebra or a standard graded $K$-algebra over a field $K$ with a canonical module $\omega_R$.  The trace of $\omega_R$ is the ideal $\tr(\omega_R)$ of $R$ which is the sum of those ideals $\varphi(\omega_R)$ with ${\varphi\in\Hom_R(\omega_R,R)}$.  The smallest number $s$ for which there exist $\varphi_1, \ldots, \varphi_s \in \Hom_R(\omega_R,R)$ with $\tr(\omega_R)=\varphi_1(\omega_R) + \cdots + \varphi_s(\omega_R)$ is called the Teter number of $R$.  We say that $R$ is of Teter type if $s = 1$.  It is shown that $R$ is not  of Teter type if $R$ is generically Gorenstein.  In the present paper, we focus especially on $0$-dimensional graded and  monomial $K$-algebras and present  various classes of such algebras  which are of Teter type.
\end{abstract}

\section*{Introduction}

Let $K$ be a field and   $R$  a Noetherian local ring or a standard graded $K$-algebra with (graded) maximal ideal $\mm$.
We will always assume that $R$ is Cohen--Macaulay and admits a canonical module $\omega_R$. The {\em trace} of an $R$-module $M$ is defined to be the ideal
\[
\tr_R(M)=\sum_{\varphi\in\Hom_R(M,R)}\varphi(M)\subseteq R.
\]
We are interested in the trace of $\omega_R$, which is called the  {\em canonical trace} of $R$.  The canonical trace determines the non-Gorenstein locus of $R$. Indeed, $R_P$ is not Gorenstein if and only if $\tr(\omega_R)\subseteq P$.

The ring $R$ is called {\em nearly  Gorenstein},  if $\mm\subseteq \tr_R(\omega_R)$. This class of rings have first been considered in \cite{HV}.  The name “nearly Gorenstein” was introduced in \cite{HHS}.  A $0$-dimensional local ring $R$ is called a {\em Teter} ring, if there exists a local Gorenstein ring $G$ such that $R\iso G/(0:\mm_G)$. This class of rings has been introduced 1974 by William  Teter \cite{T}. It has been shown (see \cite{T},\cite{HV} and \cite{ET})  that $R$ is  a Teter ring if and only if there exists an epimorphism $\varphi\: \omega_R\to \mm$. This result shows  that a Teter ring is nearly Gorenstein.

On the other hand, a $0$-dimensional local ring which is nearly Gorenstein need not be  a Teter ring, see \Cref{EXAMPLE_simplicial_complex}.  Indeed, if $R$ is nearly Gorenstein but not Gorenstein, then in general  several  $R$-module homomorphisms $\varphi\: \omega_R\to R$ are required to cover $\mm$.  More generally, we define the {\em Teter number} of $R$ to be the smallest number $s$ for which  there exist $R$-module homomorphisms $\varphi_i\: \omega_R\to R$ such that $\tr(\omega_R)=\sum_{i=1}^s\varphi_i(\omega_R)$. Thus, if $R$ is not Gorenstein, then $R$ is a Teter ring if and only if $R$ is nearly Gorenstein with Teter number $1$. This fact leads us to call a non-Gorenstein  ring $R$ to be of {\em Teter type} if its Teter number is $1$.

In Section 1 we show that if $R$ is generically Gorenstein,  then $R$ is not  of Teter type. Therefore, since we are interested in rings of Teter type, we assume throughout the rest of the paper that $\dim R=0$. Actually, we are more specific and always assume that $R$ is a $0$-dimensional local $K$-algebra, where $K$ is a field.

In Section 2 we introduce $\tau$-ideals and their companions. We call an ideal $I\subset R$ a {\em $\tau$-ideal} ($\tau$ stands for {\bf T}eter), if there exists an epimorphism $\varphi:\omega_R\to I$. Let $I\subset R$ be an ideal. A {\em companion} of $I$ is an ideal $J\subset R$ such that $J\iso I^\vee$. Here, for an $R$-module $M$ we set $M^\vee=\Hom_R(M,\omega_R)$.  In \Cref{taucompanion} it is shown that $I$ is a $\tau$-ideal if and only if $I$ has a companion $J$, in which case $J$ is also a $\tau$-ideal. An ideal $I$ is called {\em symmetric}  if there exists an isomorphism $I\iso I^\vee$. In  \Cref{tau} it is noted that any symmetric ideal is a $\tau$-ideal. The converse is  not true in general. \Cref{natural} provides a natural $\tau$-ideal of $R$ in the case that  $R=G/J$,  where $G$ is a $0$-dimensional Gorenstein $K$-algebra and $J\subset G$ is a non-zero ideal. A condition is given when this $\tau$-ideal is symmetric. This condition  is used in the proof of \Cref{completedim2},  which is one of the main results of this section. In \Cref{easy} it is shown that $R$ is of Teter type if and only if $\tr(\omega_R)$ is symmetric.  From a computational point of view, this condition is not so easy to check. But if $R$ is a $0$-dimensional standard graded  or a monomial  $K$-algebra,  more tools are available. The basic observation is formulated in \Cref{resume}. It is stated that if  $R$ is a $0$-dimensional standard (multi)graded $K$-algebra,  then $\tr(\omega_R)=\sum_{\varphi\in  \Sc}\varphi(\omega_R)$, where $\Sc$ is the set of (multi)graded   $R$-module homomorphisms $\varphi\: \omega_R\to R$ of any (multi)degree. This result is very important for the remaining sections which are devoted to the study of monomial $K$-algebras. It allows us to study ideals of Teter type in terms of the underlying divisor posets, and it suggests to call a (multi)graded ring to be of Teter type in the (multi)graded sense if there exists a (multi)graded epimorphism $\varphi\: \omega_R\to \tr(\omega_R)$. It is clear that Teter type in the multigraded sense implies Teter type in the graded sense and this  implies  Teter type in the local sense. Examples show that these implications are strict.
An observation of Vasconcelos \cite{V} implies that if $\FF$ is a free $S$-resolution of $R$ and $C$ is the kernel of $F_n\tensor_SR\to F_{n-1}\tensor_SR$, then its generators can be identified with column vectors, and the  entries of these column vectors generate the trace of $\omega_R$. In the multigraded case this fact can be used to describe an algorithm which allows to decide whether the ring is of Teter type, see \Cref{stronger}. We conclude Section 2 with \Cref{completedim2},  in which we show that if $R$ is a $0$-dimensional graded complete intersection of embedding dimension $2$ with graded maximal ideal $\mm_R$, then $R/\mm_R^k$ is of Teter type for all $k\geq 2 $ for which $\mm_R^k\neq 0$.

In Sections $3, 4$ and $5$, we focus especially on  a $0$-dimensional monomial $K$-algebra $R = S/I$, where $S = K[x_1, \ldots, x_n]$ is the polynomial ring and $I$ is a monomial ideal of $S$. We demonstrate various classes of $0$-dimensional monomial $K$-algebras of Teter type.  A $0$-dimensional monomial $K$-algebra $R = S/I$ has a canonical monomial basis $\P$, which consists of those monomials $u \in S$ with $u \not\in I$.  Then $\P$ is endowed with a structure of a poset (partially ordered set), ordered by divisibility.  The poset $\P$ is called the {\em divisor poset} of $R$.  The dual basis $\P^*$ is the monomial basis of $\omega_R$. This basis can be endowed with a poset structure which is exactly the dual of the poset structure of $\P$.  In Section $3$, the fundamental material presented in Sections $1$ and $2$ is discussed in the category of multigraded $K$-algebras in terms of poset language.  One of the crucial facts in the category of multigraded $K$-algebras is that the union of a $\tau$-ideal and any of its companions  is always symmetric (\Cref{MIGHT_BE_TRUE}).  Let $R = S/(x_1^{a_1+1}, \ldots, x_n^{a_n+1})$ be  a monomial  complete intersection, where $1 \leq a_1 \leq \cdots \leq a_n$, and $\mm_R$ is the graded maximal ideal of $R$.  One of the main results of Section $3$ is that the quotient ring $R/(0:\mm_R^k)$ is of Teter type if $k \leq a_1$ (\Cref{Chopin}).  Furthermore, when $a_1 = \cdots = a_n = 1$, it is shown that $R/(0:\mm_R^k)$ is of Teter type if and only if $k \leq n - k$ (\Cref{Mozart}).

Almost complete intersections of Teter type are studied in Section $4$.  Let $S = K[x_1, \ldots, x_n]$ and $I = (x_1^{a_1}, \ldots, x_n^{a_n}, x_1^{b_1}\cdots x_n^{b_n})$, where $b_i < a_i$ for all $i$ and $b_i > 0$ for at least two integers $i$.  In \Cref{tracemaci} it is shown that $\tr(\omega_R)$ of $R = S/I$ is generated by those monomials $x_i^{a_i - b_i}$ for which $b_i > 0$ and by $w/x_i^{b_i}$, where $w = x_1^{b_1}\cdots x_n^{b_n}$.  Our proof heavily depends on algebraic techniques.  It would be of interest to find a simple combinatorial proof of \Cref{tracemaci}.  By virtue of \Cref{tracemaci}, one can classify almost complete intersection monomial algebras of Teter type.  In fact, \Cref{compint} says that $R = S/I$ is of Teter type if and only if there exist $j \neq j'$ for which $2b_j \geq a_j$ and $2b_{j'} \geq a_{j'}$.

Finally, Section $5$ is devoted to the study of the divisor posets of simplicial complexes.  Let $\Delta$ be a simplicial complex on $[n] = \{1,2,\ldots,n\}$.  Let $S = K[x_1, \ldots, x_n]$ and $I_\Delta$ be the Stanley--Reisner ideal of $\Delta$ (\cite[pp.~16]{HH}).  We study the $0$-dimensional monomial $K$-algebra $K\{\Delta\} = S/(I_\Delta, x_1^{2}, \ldots, x_n^{2})$.  We associate each $F \subset [n]$ with the squarefree monomial $u_F = \prod_{i \in F} x_i$.  The divisor poset $\P_\Delta$ of $K\{\Delta\}$ is the finite set $\{ u_F : F \in \Delta\}$ with the partial order $\preceq$ defined by $u_F \preceq u_{F'}$ if $F' \subseteq F$.  A face $F$ of $\Delta$ is called {\em free} if there is a unique facet $F'$ of $\Delta$ with $F \subset F'$. A simplicial complex $\Delta$ is called {\em flag} if any minimal non-face of $\Delta$ is of cardinality $2$.
When $\Delta$ is flag, it is shown that $\tr(\omega_{K\{\Delta\}})$ is generated by the monomials $u_F$ for which $F$ is a free face (\Cref{trace_of_squarefree_divisor_poset}).  As a special case of flag complexes  we consider the order complex $\Delta(L)$ of a finite distributive lattice $L$.  It is proved that each facet $F$ of $\Delta(L)$ admits a face $F^\sharp \in \Delta$ with the property that $F$ is a unique facet containing $F^\sharp$. It also has the property that if $F$ is a unique facet containing a face $F' \in \Delta$, then $F^\sharp \subset F'$ (\Cref{distributive_trace}).  Thus, in particular,  the $K$-algebra $K\{\Delta(L)\}$ is of Teter type in the category of $0$-dimensional local $K$-algebras. Furthermore, we discuss independence complexes of path graphs and cycle graphs.

\section{Local rings of Teter type and Teter numbers}

In this section we show that domains  of Teter type do not exist.  Indeed, more generally  we have

\begin{Theorem}
\label{only}
If $R$ is generically Gorenstein, then $R$ is not  of Teter type. \\
\end{Theorem}

\begin{proof}
If $\dim R=0$, then by assumption $R$ is Gorenstein and hence not of Teter type. Suppose now $\dim R>0$. Then $\height(\tr(\omega_R))>0$, since $R$ is generically Gorenstein. On the contrary, assume $R$ is of Teter type. Let $P$ be a minimal prime ideal of $\tr(\omega_R)$. Then $\dim R_P>0$. Since $\omega_R$ as well as  $\tr(\omega_R)$ localize, we may replace $R$ by $R_P$,  and hence we  may assume that $R$ is a ring of Teter type  with  $\dim R>0$ and that $\tr(\omega_R)$ is $\mm$-primary. In particular, there exists an epimorphism $\varphi\: \omega_R\to \tr(\omega_R)$.  Since $\tr(\omega_R)$ is $\mm$-primary, it follows that $R_P$ is Gorenstein for all $P\neq \mm$. This implies that $\varphi_P$ is an isomorphism for all $P\neq \mm$. Thus, if  $C$ is  the kernel of $\varphi$, then $C=0$ or $\Supp(C)=\{\mm\}$. Suppose $C\neq 0$, Then $\depth C=0$, and hence also $\depth \omega_R=0$. Since $\omega_R$ is a maximal Cohen-Macaulay module, it follows that $\dim R=0$,  a contradiction.  Thus we have shown that $C=0$ which implies that $\varphi$ is an isomorphism.

 As noticed in \cite{HHS},  $\omega_R$ can be identified with an ideal in $R$, since $R$ is generically Gorenstein  and $\tr(\omega_R)=\omega_R\cdot \omega_R^{-1}$,  where  $\omega_R^{-1}=\{x\in Q(R)\: x\omega_R\subset R\}$. Here $Q(R)$ denotes the full ring of fractions of $R$. Thus we have $\omega_R\iso \omega_R\cdot \omega_R^{-1}$.  This gives us
\[
R=\omega_R: \omega_R\iso \omega_R :(\omega_R\cdot \omega_R^{-1})=(\omega_R :\omega_R):\omega_R^{-1}=R:\omega_R^{-1}=(\omega_R^{-1})^{-1}.
\]
Since $R$ is not Gorenstein, \cite[Korollar 7.29]{HK}  implies that $\omega_R\neq (\omega_R^{-1})^{-1}$, and we obtain an exact sequence $0\to  \omega_R\to (\omega_R^{-1})^{-1}\to D\to 0$ with $D\neq 0$.  Since $R$ is Gorenstein on the punctured spectrum, the inclusion map $\omega_R\to (\omega_R^{-1})^{-1}$ becomes an isomorphism after localization with $P\in \Spec(R)\setminus \mm$. This implies that $\Supp(D)=\{\mm\}$. Hence $\depth D=0$. Since $(\omega_R^{-1})^{-1}\iso R$, it follows then that $\depth \omega_R=1$. Thus, $\dim R=1$.

Since $\omega_R\subset (\omega_R^{-1})^{-1}\iso R$, we deduce  that $(\omega_R^{-1})^{-1} =(x)$ for some  $x\in R$  and  that $\omega_R\subset (x)$. Hence, $\omega_R=xL$ for some ideal $L$. Since $\Ann(\omega_R)=(0)$, the element $x$ must be a non-zerodivisor. Therefore, $L$ is also a canonical  ideal of $R$. Moreover, $(x)=R:(R:xL) =x(R:(R:L))$. This shows that $R:(R:L)=R$. Hence, replacing $\omega_R$ by $L$, we may assume that $R=(\omega_R^{-1})^{-1}$. Then $\omega_R^{-1} =((\omega_R^{-1})^{-1})^{-1}=R^{-1}=R$.

The exact sequence $0\to \omega_R\to R \to R/\omega_R\to\ 0$ induces the exact sequence
\[
0\to R\to \omega_R^{-1} \to \Ext_R^1(R/\omega_R, R)\to 0.
\]
The map  $R\to \omega_R^{-1}$ is the identity, since $\omega_R^{-1}=R$. Therefore, $\Ext_R^1(R/\omega_R, R)=0$. Let $y\in \omega_R$ be a non-zerodivisor. Set $\overline{R}=R/(y)$. The exact sequence
\[
0\to R\stackrel{y}{\to} R\to \overline{R}\to 0
\]
induces the exact sequence
\[
0\to \Hom_R(R/\omega_R, \overline{R})\to \Ext_R^1(R/\omega_R, R)\to \cdots,
\]
since $\Hom_R(R/\omega_R,R)=0$. It follows that $\Hom_R(R/\omega_R, \overline{R})=0$. Notice that  $R/\omega_R$  is an $\overline{R}$-module, which implies that  $\Hom_R(R/\omega_R, \overline{R})\iso
\Hom_{\overline{R}}(R/\omega_R, \overline{R})$. Thus, $\Hom_{\overline{R}}(R/\omega_R, \overline{R})=0$.  Since $\dim \overline{R}=0$, there exists $c\in \overline{R}$ with $c\neq 0$ and $c\overline{\mm}=0$. Here $\overline{\mm}$ denotes the  maximal ideal of $\overline{R}$. The non-zero $\overline{R}$-module homomorphism $\overline{R}\to \overline{R}$, $1\mapsto c$ factorizes over $R/\omega_R$. Thus, $\Hom_{\overline{R}}(R/\omega_R, \overline{R})\neq 0$, a contradiction.
\end{proof}

\begin{Corollary}
\label{two}
Suppose that $R$ is generically Gorenstein but not Gorenstein.  Then the Teter number of $R$ is $\geq 2$.
\end{Corollary}

For an ideal $I\subset R$ we denote by $\mu(I)$ the minimal number of generators of $I$. We have

\begin{Lemma}
\label{cover}
The Teter number of $R$ is $\leq \mu(\tr(\omega_R))$.
\end{Lemma}

\begin{proof} Let $\A=\Hom_R(\omega_R,R)$. Let $\varphi\in \A$.  We denote by  $V_\varphi$ the $K$-subspace $(\varphi(\omega_R)+\mm\tr(\omega_R))/\mm\tr(\omega_R)$ of the $K$-vector space $\tr(\omega_R)/\mm\tr(\omega_R)$.

We first  prove   by induction on $n$ the  following  elementary fact:   if $V$ is any $K$-vector space of dimension $n$ and $V=\sum_{\lambda\in\Lambda}V_\lambda$, then there exist $\lambda_1,\ldots,\lambda_n\in \Lambda$ such that $V=\sum_{i=1}^nV_{\lambda_i}$. The assertion is trivial for $n=1$. Now let us assume that $n>1$. There exists $\lambda_1\in \Lambda$ such that $V_{\lambda_1}\neq 0$. Let $W=V/V_{\varphi_1}$. Then $W$ is a $K$-vector space of dimension $<n$, and $\sum_{\lambda\in \Lambda, \lambda\neq \lambda_1}(V_\lambda+V_{\lambda_1})/V_{\lambda_1}=W$. By the induction hypothesis, there exist $\lambda_2,\ldots,\lambda_n$ such that $\sum_{i=2}^n(V_{\lambda_i}+V_{\lambda_1})/V_{\lambda_1}=W$. It follows that
$V=\sum_{i=1}^nV_{\lambda_i}$.

Applying this result to our particular case, we see that there exist $\varphi_1,\ldots,\varphi_n$ such that $(\sum_{i=1}^n\varphi_i(\omega_R))+\mm\tr(\omega_R)=\tr(\omega_R)$. Nakayama's Lemma implies that $\sum_{i=1}^n\varphi_i(\omega_R)=\tr(\omega_R)$.
\end{proof}

\section{$\tau$-ideals and their companions for $0$-dimensional local $K$-algebras}
\label{abstract}

Let $K$ be a field, and let $R$ be a $0$-dimensional local $K$-algebra. Then $R$ admits a canonical module $\omega_R$. In fact, $\omega_R=R^\vee$, where $R^\vee=\Hom_K(R,K)$. Let $M$ be a finitely generated $R$-module. Then we set $M^\vee=\Hom_K(M,K)$. One has
\[
M^\vee =\Hom_R(M,\omega_R) \quad \text{and} \quad (M^\vee)^\vee =M.
\]
\begin{Remark}\label{Juergen}
It is well known, see for example \cite[Satz 6.10]{HK}, that the minimal number of generators $\mu(M)$ of $M$ is equal to the socle dimension $\sigma(M^{\vee})$ of $M^{\vee}$.
\end{Remark}

Let $I\subset R$ be an ideal. A {\em companion} of $I$ is an ideal $J\subset R$ such that $J\iso I^\vee$. An ideal $I$ is called a {\em $\tau$-ideal}, if there exists an $R$-module homomorphism $\varphi\:\; \omega_R\to R$ with $\varphi(\omega_R)=I$.

\begin{Lemma}
\label{taucompanion}
Let $I\subset R$ be an ideal. Then $I$ is a $\tau$-ideal if and only if $I$ has a companion $J$. In that case $J$ is also a $\tau$-ideal.
\end{Lemma}

\begin{proof}
Suppose $I$ is a $\tau$-ideal. Then there exists an epimorphism $\varphi:\omega_R\to I$. This epimorphism induces a monomorphism $\varphi^\vee\: I^\vee \to R$. Let $J$ be the image of $I^\vee \subseteq R$. Then $J$ is an ideal in $R$ and $J$ is isomorphic to $I^\vee$. Hence $J$ is a companion of $I$.

Let $J$ be a companion of $I$. Then there exists an isomorphism $\alpha: I^\vee\to J$. The inclusion map $I\to R$ induces an epimorphism $\psi\: \omega_R\to I^\vee$. Then $\varphi=\alpha \circ \psi: \omega_R \to J$ is an epimorphism. Hence $J$ is a $\tau$-ideal.

Moreover, if $I$ has a companion $J'$, then  $J'\iso I^\vee$, and so  $I\iso (J')^\vee$. Therefore, $I$ is a companion of $J'$. Hence $I$ is a $\tau$-ideal.
\end{proof}

An ideal $I\subset R$ is called {\em symmetric}, if there exists an isomorphism $\gamma\:\; I\to I^\vee$. In other words, $I$ is symmetric if $I$ is a companion of $I$.

\begin{Corollary}
\label{tau}
Let $I\subseteq R$ be a symmetric ideal. Then $I$ is a $\tau$-ideal.
\end{Corollary}

Not all $\tau$-ideals are symmetric, see Example~\ref{EXAMPLE_Z}.

\begin{Proposition}
\label{natural}
Let $G$ be a $0$-dimensional local Gorenstein $K$-algebra, and let $J\subset G$ be a non-zero ideal and $R=G/J$. Then $I=((0:J)+J)/J\subset R$ is a $\tau$-ideal. Moreover,  $I$ is symmetric if $J\subset 0:J$.
\end{Proposition}

\begin{proof}
Note that $\omega_R=\Hom_G(R,G)\iso 0:J$. The composition of the inclusion map $0:J\subset G$ with the canonical epimorphism $G\to R$ yields an $R$-module homomorphism $\omega_R\to R$ whose image is $I=((0:J)+J)/J$. Thus, $I$ is a $\tau$-ideal.

Assume now that $J\subset 0:J$. Then $I=(0:J)/J$, and we have the exact sequence
\[
0\to (0:J)/J\to  G/J\to G/(0:J)\to 0.
\]
Since $G$ is Gorenstein, the functor $\Hom_G(-, G)$ is exact.  Applying this functor to the above exact sequence, we obtain the exact sequence
\[
0\to\Hom_G(G/(0:J), G)  \to  \Hom_G(G/J, G)\to\Hom_G((0:J)/J, G)\to 0.
\]

Observe that  $\Hom_G(G/(0:J), G)\iso 0:(0:J)=J$, since $G$ is Gorenstein. Furthermore, $\Hom_G(G/J, G)\iso 0:J$ and $\Hom_G((0:J)/J, G)=
((0:J)/J)^\vee$. This shows that $(0:J)/J\iso ((0:J)/J)^\vee$, as desired.
\end{proof}

Recall that $R$ is of Teter type, if  there exists an epimorphism  $\varphi\: \omega_R\to \tr(\omega_R)$.

\begin{Proposition}
\label{easy}
The following conditions are equivalent:
\begin{enumerate}
\item[(a)] $R$ is of Teter type.
\item[(b)] $\tr(\omega_R)$ is symmetric.
\end{enumerate}
\end{Proposition}

\begin{proof}
(a)\implies (b): If $R$ is of Teter type, then $I=\tr(\omega_R)$ is a $\tau$-ideal, and so $I$  has a companion $J$. By \Cref{taucompanion},  $J$ is also a $\tau$-ideal, which implies  that $J\subseteq \tr(\omega_R)=I$. Since $J\iso I^\vee$, it follows that $\dim_KJ=\dim_KI$. Thus we conclude that $J=I$. This means that $I$ is symmetric.

(b)\implies (a):  $\tr(\omega_R)$ is symmetric, then $\tr(\omega_R)$ is a $\tau$-ideal, see \Cref{tau}. This means that $R$ is of Teter type.
\end{proof}

We assume for a moment  that $R$ is a local ring and that  $M$ is a finitely generated  $R$-module. As observed by Vasconcelos \cite{V}, the trace of a module $M$ can be computed as follows: we choose  a  free presentation
\[
G\stackrel{A}{\to}F\to M\to 0
\]
with finitely generated free $R$-modules $F$ and $G$, where $A$ is the matrix describing the $R$-module homomorphism $G\to F$ with respect to the basis   $f_1,\ldots, f_n$ of $F$ and the basis  $g_1,\ldots, g_m$  of $G$. Note that $A$ is an $n\times m$ matrix.  Let $\bb=(b_1,\ldots,b_n)$ be a vector with entries in $R$. Then $\bb$ defines an $R$-module homomorphism $\varphi\: M\to R$ if and only if $\bb A=0$. Thus $\tr(M)$ is generated by the entries of all the vectors $\bb$  with $\bb A=0$, equivalently, by all column vectors $\bb^{\sf T}$ such that $A^{\sf T}\bb^{\sf T}=0$. The corresponding statements hold for graded rings and graded modules.

Assume now that $R=S/I$ is a $0$-dimensional standard graded $K$-algebra with $S=K[x_1,\ldots,x_n]$ the polynomial ring.  Let
\[
0\to F_n \to F_{n-1} \to \ldots \to F_1\to F_0\to R\to 0.
\]
We do not insist that this resolution is minimal, but only require that the length of this resolution coincides with the projective dimension of $R$ (which is equal to $n$).  We denote by $N^*$ the $S$-dual of an $S$-module $N$. It is well-known (see \cite{BH}) that the $S$-dual of  $F_n \to F_{n-1}$ provides a graded free  presentation  of the $S$-module $\omega_R$. That is, we have an exact sequence
$
F_{n-1}^*\to F_{n}^*\to \omega_R.
$
Tensorizing this sequence of $S$-modules with $R$, we obtain a free $R$-module presentation  of $\omega_R$,  and we may apply the above mentioned observation of Vasconcelos.
Let $e_1,\ldots,e_m$ be a basis of $F_n$, which is also a basis of $F_n\tensor_SR$. Then for $F_n^*\tensor_SR$ we have the dual basis $e_1^*,\ldots,e_m^*$ and an epimorphism $\epsilon: F_n^*\tensor_SR\to \omega_R$. Let  $\omega_i=\epsilon(e_i^*)$. Then an element $\cb=\sum_{i=1}^m c_ie_i$ belongs to $C=\Ker(F_n\tensor_SR\to F_{n-1}\tensor_SR)$ if and only if the assignment $\varphi(\omega_i)=c_i$ for $1\leq i\leq m$ defines an $R$-module homomorphism $\varphi: \omega_R\to R$.

Let $\cb_1,\ldots,\cb_r$ be a generating set of $C$ with $\cb_j=(c_{1j},\ldots,c_{mj})^{\sf T}$. Consider the $m\times r$ matrix $E=(c_{ij})$ which we denote again by $C$. The discussion above shows that each of the column vectors $\cb_j$ of $E$ defines an  $R$-module homomorphism $\varphi_j:\omega_R\to R$ and any other $R$-module homomorphism
$\varphi:\omega_R\to R$ is an $R$-linear combination of $\varphi_1,\ldots,\varphi_r$. Hence we conclude

\begin{Proposition}
\label{vasconcelos}
Let $E$ be a matrix representing the kernel of $F_n\tensor_SR\to F_{n-1}\tensor_SR$. Then the entries of $E$ generate $\tr(\omega_R)$.
\end{Proposition}

Assume now that $R=S/I$ is a $0$-dimensional standard (multi)graded $K$-algebra. Then $R$ admits a (multi)graded free $S$-resolution $\FF$ of length $n$. Moreover, $\omega_R$ admits a (multi)graded free $S$-resolution $\FF^*$. Let $e_1,\ldots,e_r$ be a homogeneous basis of $F_n$. A column vector $\cb=(c_1,\ldots,c_r)^{\sf T}\in C$ is homogenous if $\deg(c_i)+\deg(e_i)$ does not depend on $i$ for all $i$ with $c_i\neq 0$. This common value is called the degree of $\cb$.  The homogeneous column vectors generate $C$.
The map $\varphi:\omega_R\to R$ corresponding to a given homogeneous column vector $\cb$ has (multi)degree $\deg(\cb)$.
Resuming what we discussed so far we have

 \begin{Corollary}
 \label{resume}
 Let $R$ be a $0$-dimensional standard (multi)graded $K$-algebra. Then $\tr(\omega_R)=\sum_{\varphi\in  \Sc}\varphi(\omega_R)$, where $\Sc$ is the set of graded, resp.\  multigraded $R$-module homomorphisms $\varphi\: \omega_R\to R$ (not necessarily of degree zero).
\end{Corollary}

Suppose again  that $R$ is graded. Then  we say that $I\subset R$ is a {\em $\tau$-ideal} if there exists a graded $R$-module homomorphism $\varphi\: \omega_R\to R$ with $I=\varphi(\omega_R)$. Similarly, if $M$ is a finitely generated graded $R$-module, then $M^\vee$ is  defined to be the module of graded $R$-module homomorphisms $\Hom_K(M,K)$. In particular, there is a graded version for all the concepts and the results introduced before. Similar statements apply when $R$ is a monomial $K$-algebra.

\medskip
Note that any $0$-dimensional monomial $K$-algebra is graded, and any graded $K$-algebra is local. Thus we have:
 \[
 \text{Teter type in the $\ZZ^n$-graded sense}\Rightarrow \text{Teter type in the graded sense}\Rightarrow \text{Teter type}.
 \]
 These implications are strict, as is shown in \Cref{Tschaikovsky}.

\medskip
 With the assumptions and notation of \Cref{resume} we have

 \begin{Corollary}
 \label{whether or not}
 The $K$-algebra $R$ is of Teter type, if and only if there exists a homogeneous column vector $\cb\in C$  such  that the entries of $\cb$ and the entries of $C$ generate the same ideal.
\end{Corollary}

A stronger statement as the one given in \Cref{whether or not} can be made when $R$ is a $0$-dimensional monomial $K$-algebra.

 \begin{Corollary}
 \label{stronger}
 Assume that $K$ is an infinite field, and let $R$ be $0$-dimensional monomial $K$-algebra.  Let $\cb_1,\ldots,\cb_m$ be a system of generators of $C$ and let $E$ be the matrix whose column vectors are $\cb_1,\ldots\cb_m$.  For each $\ab\in\ZZ^n$, let  $E_\ab$ be the submatrix of $E$   consisting of the column vectors $\cb_i$ with $\deg(\cb_i)=\ab$. Then $R$ is of Teter type if and only  for some $\ab\in \ZZ^n$,  the entries of $E_\ab$  and the entries of $E$ generate the same ideal.
 \end{Corollary}

 \begin{proof}
 For  given  $\ab\in \ZZ^n$, let $\cb_{i_1},\ldots,\cb_{i_k}$ be the column vectors of $E_\ab$. Since these column vectors are homogeneous, it follows that $\cb_{i_j}=(\lambda_{1j}u_1,\ldots,\lambda_{rj}u_r)^{\sf T}$ for $j=1,\ldots,k$ with $\lambda_{sj}\in K$ for $s=1,\ldots, r$. Since $K$ is infinite,  we may choose $\mu_1,\ldots, \mu_k\in K$ such that $\sum_{j=1}^k\mu_j\lambda_{sj}\neq 0$ for those $s$ for which there exists $j$ with $\lambda_{sj}\neq 0$. Then the  ideal generated by the entries of the column vector $\cb=\sum_{j=1}^k\mu_j\cb_{i_j}$ is the same as the ideal generated by the entries of $E_\ab$, and the desired conclusion follows from \Cref{whether  or not}.
\end{proof}

Let $\FF$ is a minimal (multi)graded free resolution of $R$ of length $n$, as before. Then in the graded case, $\Soc(R)$ is generated by polynomials $f_1,\ldots,f_r$ with $\deg(f_i)=\deg(e_i)-n$ and in the multigraded case $\Soc(R)$ is generated by the monomials $x^{\ab_i}/(x_1\cdots x_n)$, where $\ab_i=\deg(e_i)$.

 Note that \Cref{stronger}, compared with \Cref{whether or not},  has  the advantage that we do not have to search for a suitable column vector to check whether $R$ is of Teter type, but only need to look at the finitely many graded components of $C$. The following examples demonstrate the method.

\begin{Example}
Consider $R=K[x,y]/(x^4,y^4,x^2y^2)$. Computations in Macaulay2 show that the corresponding matrix is
$ E=
\begin{pmatrix}
y^2 & 0 & 0 & x^2\\
0 & x^2 & y^2 & 0
\end{pmatrix}
$.
Let $\FF$ be the multigraded minimal free resolution of $R$. Since $\Soc(R)=(x^3y,xy^3)$, $F_2$ has a  homogeneous  basis $e_1$ and $e_2$ with $\deg(e_1)=(4,2)$ and $\deg(e_2)=(2,4)$.

$$
\begin{pmatrix}
x^4y^2 & x^2y^4\\

\end{pmatrix}
\cdot
\begin{pmatrix}
y^2 & 0 & 0 & x^2\\
0 & x^2 & y^2 & 0

\end{pmatrix}
=
\begin{pmatrix}

x^4y^4 & x^4y^4 & x^2y^6 & x^6y^2\\
\end{pmatrix}
,$$
which means that the maps given by the first and the second column of $E$ both have multidegree $x^4y^4$, the map given by the third column of $M$ has multidegree $x^2y^6$ and the map given by the fourth column of $E$ has multidegree $x^6y^2$. From $E$ we also know that $\tr(\omega_R)=(x^2,y^2)$. Note that the ideal generated by monomials from the first and the second column of $E$ is exactly $(x^2,y^2)$, which proves that $R$ is of Teter type.
\end{Example}

\begin{Example}
Consider $R=K[x,y]/(x^3,y^3,xy)$. Computations in Macaulay2 show that the corresponding matrix is
$ E=
\begin{pmatrix}
y^2 & 0 & 0 & x\\
0 & x^2 & y & 0

\end{pmatrix}
$.
Let $\FF$ be the multigraded minimal free resolution of $R$. Since $\Soc(R)=(x^2,y^2)$, $F_2$ has a  homogeneous  basis $e_1$ and $e_2$ with $\deg(e_1)=(3,1)$ and $\deg(e_2)=(1,3)$.

$$
\begin{pmatrix}
x^3y & xy^3\\

\end{pmatrix}
\cdot
\begin{pmatrix}
y^2 & 0 & 0 & x\\
0 & x^2 & y & 0

\end{pmatrix}
=
\begin{pmatrix}

x^3y^3 & x^3y^3 & xy^4 & x^4y\\

\end{pmatrix}
,$$
which means that the maps given by the first and the second column of $E$ both have multidegree $x^3y^3$, the map given by the third column of $E$ has multidegree $xy^4$ and the map given by the fourth column of $E$ has multidegree $x^4y$. From $E$ we also know that $\tr(\omega_R)=(x,y)$. Note that the ideal generated by monomials from the first and the second column of $E$ is $(x^2,y^2)$; the ideal generated by monomials in the third column of $E$ is $(y)$ and the ideal generated by monomials in the fourth column of $E$ is $(x)$. This proves that $R$ is not of Teter type.
\end{Example}

\bigskip
We apply the theory, as developed so far,  to show

\begin{Theorem}
\label{completedim2}
Let $f, g\in S=K[x,y]$ be a regular sequence of homogeneous polynomials,  and let $G=S/(f,g)$ be the complete intersection ring with graded maximal ideal $\nn$.
Assume further that $K$ is infinite or that $f$ and $g$ are monomials. Then for each integer  $1\leq k \leq a+b-2$,  the $K$-algebra $R=G/\nn^k$ is of Teter type.
\end{Theorem}

\begin{proof}
Let $a=\deg(f)$ and $b=\deg(g)$. We may assume that $a\leq b$. Let $\mm_R$ be the maximal ideal of $R$. We claim that
\[
\tr(\omega_R)=
\begin{cases}
\mm_R^{k-1} & \text{if } k\leq a\\
\mm_R^{a-1} & \text{if } a<k\leq b\\
\mm_R^{a+b-1-k} & \text{if } k>b,
\end{cases}
\]
and that $\tr(\omega_R)$ is symmetric in all these cases. This then shows that $R$ is of Teter type.

Assume first that $k\leq a$. Then $R=S/(x,y)^k$,and
\[
0\to S(-k-1)^k\stackrel{A}{\To}S(-k)^{k+1}\to (x,y)^k\to 0.
\]
is the resolution of $(x,y)^k$. Here $A$ is the $(k+1)\times k$-matrix
\[
\begin{pmatrix}
y &  0& \cdots&\cdots  & 0\\
-x & y& \cdots & \cdots & 0\\
0  & -x &\ddots &\cdots& 0\\
\vdots& \vdots& \ddots &\ddots &\vdots\\
0& 0& \cdots&\ddots  & y\\
0& 0& \cdots&\cdots  & -x\\
\end{pmatrix}
\]
By \Cref{vasconcelos}, the entries of the kernel of $R^k\stackrel{A}{\To}R^{k+1}$ generate $\tr(\omega_R)$. Let $(f_1,\ldots,f_k)^{\sf T}$ be in the kernel of this map. We may assume that all $f_i$ are homogeneous. Since  $yf_1=0$, it follows that $\deg(f_1)\geq k-1$,  and hence $f_1\in \mm_R^{k-1}$. Next we have $-xf_1+yf_2=0$. Since $xf_1=0$, it follows that $yf_2=0$. As before, we deduce that $f_2\in \mm_R^{k-1}$. Proceeding in the way, we see that $f_i\in \mm_R^{k-1}$ for all $i$. This shows that $\tr(\omega_R)=\mm_R^{k-1}$. Note that $\mm_R^{k-1}=\Soc(R)$. It is obvious that $\Soc(R)$ is a symmetric ideal.

In the next steps we show that $\mm_R^{a-1}$ is a symmetric $\tau$-ideal of $R$ if $a<k\leq b$, and that $\mm_R^{a+b-k-1}$ is a symmetric $\tau$-ideal of $R$ if $k>b$. Indeed, if  $a<k\leq b$, then $R=S/(f, (x,y)^k)$. Hence we may write $R =\widetilde{G}/\tilde{\nn}^k$, where $\widetilde{G}=S/(f,p)$ is a complete intersection and $p$ is a minimal generator of $(x,y)^k$ and $\tilde{\nn}$ is the maximal ideal of $\tilde{G}$.
Indeed, if $K$ is infinite, then a non-zerodivisor $p$ modulo $f$ can be found among the generators of $(x,y)^k$. On the other hand, if $f$ and $g$ are monomials, then each of them is a pure power, say $f=x^a$. In this case we may choose $p=y^k$.
The  socle degree of this complete intersection is $a+k-2$. It follows that $0: \tilde{\nn}^k=\tilde{\nn}^{a-1}$. Since $\tilde{\nn}^k\subset 0: \tilde{\nn}^k$, the desired conclusion follows from \Cref{natural}.
Next assume $k>b$. Then in the ring $G$ we have $0:\nn^k=\nn^{a+b-1-k}$, because  the socle degree of $G$ is $a+b-2$. Since $\nn^k\subset \nn^{a+b-1-k}$, the desired conclusion follows again from \Cref{natural}.

It remains to be shown that any $\tau$-ideal is contained in $\mm_R^{a-1}$ if $a<k\leq b$, and is contained in $\mm_R^{a+b-1-k}$ if $k>b$.

Let us first treat the case in which  $a<k\leq b$. Then $R=S/(f,\mm^k)$, and since $\mm^k:(f)=\mm^{k-a}$, we obtain the exact sequence
\[
0\to (S/(\mm^{k-a}))(-a)\to S/\mm^k\to R\to 0.
\]
Let $\FF$ be the minimal graded resolution of $\mm^k$ and $\GG$ be the minimal graded free resolution of $(S/(\mm^{k-a}))(-a)$. The multiplication map given by $f$ can be lifted to a graded complex homomorphism $\GG\to \FF$, so that we obtain a
commutative diagram

\bigskip
\begin{tikzcd}
\GG\: & 0\arrow[r]& S(-k-1)^{k-a}\arrow[d, "\beta"']\arrow[r]&S(-k)^{k-a+1}\arrow[d,"\alpha"']\arrow[r] & S(-a)\arrow[d, "f"']\arrow[r]& 0\\
\FF\: &0\arrow[r]& S(-k-1)^k \arrow[r]& S(-k)^{k+1}\arrow[r] & S\arrow[r] &0.
\end{tikzcd}

\noindent
Since the multiplication map $S(-a)\stackrel{f}{\To} S$ is injective, it follows  that for any element $g\in (G_1)_k$ we have that $\alpha(g)\neq 0$. From this it can be seen that $\alpha$  is injective. Similarly, $\beta$ is injective.   In fact, $\alpha$ and $\beta$ are  split injective,  because by degree reasons, their image  is a direct summand of  $S(-k)^{k+1}$ and $S(-k-1)^{k}$, respectively.

Let $\DD$ be the mapping cone of $\GG\to \FF$. Then we obtain the short exact sequence of complexes
\[
0\To \FF\tensor_SR\To \DD\tensor_SR\To \GG[-1]\tensor_SR\To 0,
\]
where $\GG[-1]$ is the complex $\GG$ with the homological shift $-1$. In other words, $\GG[-1]_i=G_{i-1}$ for all $i$. This short exact sequence give rise to the long exact sequence
\[
\cdots \to H_2(\FF\tensor_SR)\To H_2(\DD\tensor_SR)\To H_1(\GG\tensor_SR)\to H_1(\FF\tensor_SR)\to \cdots
\]
The map $H_1(\GG\tensor_SR)\to H_1(\FF\tensor_SR)$ is a connecting homomorphism in the long exact sequence and  is induced by $\alpha\tensor_S R$. Since $\alpha$ is split injective, it follows that $\alpha\tensor_S R$ is split exact as well. Therefore, $H_1(\GG\tensor_SR)\to H_1(\FF\tensor_SR)$ is injective. This implies that $H_2(\FF\tensor_SR)\to H_2(\DD\tensor_S R)$ is surjective. Since $\DD$ is a graded free $S$-resolution of $R$, it follows that $H_2(\DD\tensor_S R)\iso \Tor^S_2(R,R)$ as graded $R$-modules. Hence we have a graded epimorphism $H_2(\FF\tensor_SR)\to \Tor^S_2(R,R)$.

The graded minimal free $S$-resolution of $R$ is given as  $0\to F_2/G_2\to F_1/G_1\dirsum S(-a)\to S\to 0$, and hence it has the form
\[
\overline{\DD}\: 0\to S(-k-1)^a\to S(-k)^a\dirsum S(-a)\to S\to 0.
\]
This yields the exact sequence $0\to \Tor^S_2(R,R)\to  R(-k-1)^a\to R(-k)^a\dirsum R(-a)$, and by using the result of Vasconcelos we need to show that  $\Tor^S_2(R,R)_{d+k+1}=0$ for $d<a-1$. For this it suffices to show that $H_2(\FF\tensor_SR)_{d+k+1}=0$ for $d<a-1$.

Note that $H_2(\FF\tensor_S R)$ is the kernel of $R(-k-1)^k\to R(-k)^{k+1}$. Hence we get the exact sequence
$
H_2(\FF\tensor_S R)_{d+k+1}\to R_d^k\to R_{d+1}^k.
$
Since $d<a-1$, it follows that $R_d=S_d$ and $R_{d+1}=S_{d+1}$, so that $0\to H_2(\FF\tensor_S R)_{d+k+1}\to S_d^k\to S_{d+1}^k$ is exact. Since $S(-k-1)^k\to S(-k)^{k+1}$ is injective, we deduce that $H_2(\FF\tensor_S R)_{d+k+1}=0$, as desired.

Now we deal with the case $k>b$. If $k=a+b-2$,  then $R$ is even a Teter ring, and we are done. Thus we may now assume that $b<k\leq a+b-1$.    Then  we have  the following  exact sequence
\[
0\to (S/(f,\mm^k):(g))(-b)\to S/(f,\mm^k)\to R\to 0.
\]
Let $h\in (f,\mm^k):(g)$ be a homogeneous polynomial with $\deg(h)=c$.  Then $hg\in (f,\mm^k)$. If $c+b<k$, then $hg=h'f$ for some $h'\in S$. Since $f,g$ is a regular sequence, it follows then that $f$ divides $h$, and this implies that $h\in (f,\mm^k)$. On the other hand, if $c+b\geq k$, then $hg\in \mm^k$. Therefore, $h\in \mm^{k-b}$, and hence  $(f,\mm^k):(g)=(f,\mm^{k-b})$. Since $k\leq a+b-1$, it follows that  $(f,\mm^k):(g)=(f,\mm^{k-b})=\mm^{k-b}$, and we obtain the  exact sequence
\[
0\to (S/(\mm^{{k-b}}))(-b)\to S/(f,\mm^k)\to R\to 0.
\]
The map $(S/(\mm^{k-b}))(-b)\stackrel{g}{\to}  S/(f,\mm^k)$ can be extended to a complex homomorphism $\HH\to \DD$ with $\alpha\: H_1\to D_1$ and $\beta\: H_2\to D_2$ split injective.  Here,
$$\HH: 0\to (-k-1)^{k-b} \to S(-k)^{k-b+1}\to S(-1)\to 0$$ is the free $S$-resolution of  $(S/(\mm^{k-b}))(-b)$ and  $\DD$ is the mapping cone from before.  Composing $\HH\to \DD$ with the  complex homomorphism $\DD\to \overline{\DD}$ onto the graded minimal free $S$-resolution  of $S/(f,\mm^k)$, we obtain   a commutative diagram

\bigskip

\begin{tikzcd}
\HH\:& 0\arrow[r]& S(-k-1)^{k-b}\arrow[d, "\beta'"']\arrow[r]&S(-k)^{k-b+1}\arrow[d,"\alpha'"']\arrow[r] & S(-b)\arrow[d, "g"']\arrow[r]& 0\\
\overline{\DD}\:& 0\arrow[r]& S(-k-1)^a \arrow[r]& S(-k)^{a}\dirsum S(-a)\arrow[r] & S\arrow[r] &0
\end{tikzcd}

\medskip
\noindent
with $\alpha'$ and $\beta'$ split injective. This yields a minimal  free $S$-resolution of $R$ which is of the form
\[
0\to S(-k-1)^{a+b-k}\to S(-k)^{a+b-1-k}\dirsum S(-b)\dirsum S(-a)\to S\to 0.
\]
Let $\EE$ be the mapping cone of the complex homomorphism $\HH\to \DD$. Then we obtain the short exact sequence of complexes
\[
0\to \DD\tensor_SR\to \EE\tensor_SR\to \HH[-1]\tensor_SR\to 0
\]
which gives rise to the long exact sequence
\[
\cdots \to H_2(\DD\tensor_SR)\To H_2(\EE\tensor_SR)\To H_1(\HH\tensor_SR)\to H_1(\DD\tensor_SR)\to \cdots
\]
The map $H_1(\HH\tensor_SR)\to H_1(\DD\tensor_SR)$ is a connecting homomorphism in the long exact sequence which is induced by $\alpha\tensor_S R$. Since $\alpha$ is split injective, it follows that $\alpha\tensor_S R$ is split exact as well. Therefore, $H_1(\HH\tensor_SR)\to H_1(\DD\tensor_S R)$ is injective. This implies that $H_2(\DD\tensor_SR)\to H_2(\EE\tensor_S R)$ is surjective. Since $\EE\tensor_SR\iso R(-k-1)^{a+b-k}$, we need to show that $H_2(\EE\tensor_SR)_{d+k+1}=0$ for $d<a+b-1-k$. We have the graded epimorphisms  $H_2(\DD\tensor_SR)\to H_2(\EE\tensor_S R)$ and $H_2(\FF\tensor_SR)\to H_2(\DD\tensor_S R)$. Thus it suffices to show that $H_2(\FF\tensor_SR)_{d+k+1}=0$ for $d<a+b-1-k$. Note that $R(-k-1)_{d+k+1}=R_d=S_d$ and that  $R(-k)_{d+k+1}=R_d=S_{d+1}$,  since $d<a+b-k-1<a-1$. The last equation follows since
$k>b$.  Thus we have an exact sequence $0\to H_2(\FF\tensor_SR)_{d+k+1}\to S^k_d\to S^{k+1}_{d+1}$. It follows that $ H_2(\FF\tensor_SR)_{d+k+1}=0$, because  $S(-k-1)^k\to S(-k)^{k+1}$ is exact.
\end{proof}

We expect that if $G$ is a $0$-dimensional local Gorenstein $K$-algebra with the maximal ideal $\mm$, then $G/\mm^k$ is of Teter type for all $k>1$ for which $\mm^k\neq 0$. The converse however is not true in general. Indeed, there is a $K$-algebra $R$ (see \Cref{TRACE_NOT_CONNECTED_TETER_TYPE}) which is of Teter type whose Hilbert function is given by the sequence $(1,2,3,4,2)$. Suppose $R$ is of the form $G/\mm^k$, where $G$ is a Gorenstein ring. Then $\embdim(G)=2$ and hence $G$ is a complete intersection. This implies that any two consecutive integers in the sequence describing the Hilbert function of $G$ differ at most by one. Hence if $R$ of the form $G/\mm^k$, then a similar property should hold for the Hilbert function of $R$, which is not the case in our example.

\section{Monomial $K$-algebras of Teter type}

Let $K$ be a field, and let $R$ be a $0$-dimensional standard graded $K$-algebra with graded maximal ideal $\mm$.  The $R$-module structure of $\omega_R=\Hom_K(R,K)$ is given as follows: for $a\in R$ and for $\varphi\in \Hom_K(R,K)$, one defines $a\varphi\in \Hom_K(R,K)$ by setting $(a\varphi)(b)=\varphi(ab)$ for $b\in R$.  The $K$-algebra $R$ admits a monomial $K$-basis $\B$. The dual basis $\B^*$ of $\B$ consists of those elements $u^*$ with $u\in\B$, where for $v\in \B$,
\[
u^*(v)=
\begin{cases}
1 & \text{if  $v=u$},  \\

0 & \text{if  $v\neq u$}.
\end{cases}
\]

A $K$-algebra $R=S/I$, where $S = K[x_1, \ldots, x_n]$ is the polynomial ring, is called a {\em monomial $K$-algebra} if $I \subseteq S$ is a monomial ideal.  A monomial $K$-algebra has a canonical monomial $K$-basis, namely $\P=\{u+I\: u\in S\setminus I \, \, \text{is a monomial}\}$.  One denotes the residue class modulo $I$ of a monomial $u\not\in  I$ again by $u$. Then the canonical monomial $K$-basis $\P$ of $R$ consists of those monomials $u\in S$ which do not belong to $I$. For the rest of this section $R$ will denote a $0$-dimensional monomial $K$-algebra.

It is noted that $R$ and $\omega_R$ are $\ZZ^n$-graded $R$-modules. Let $\ab\in \ZZ^n$. Then
\[
R_\ab=
\begin{cases}
 K\xb^{\ab}  &\text{if  $\xb^{\ab}\in \P$},  \\

0 & \text{otherwise},
\end{cases}
\]
and
\[
(\omega_R)_\ab=
\begin{cases}
 K(\xb^{-\ab})^* &  \text{if  $\xb^{-\ab}\in \P$},  \\
0 & \text{otherwise}.
\end{cases}
\]
One then defines $\deg(\xb^{\ab}) = \ab$ for $\xb^{\ab}\in \P$ and $\deg((\xb^{\ab})^*) = -\ab$ for $(\xb^{\ab})^*\in \P^*$.

\medskip
In this section we are interested in $0$-dimensional multigraded $K$-algebras. Unless otherwise stated, the concepts like $\tau$-ideals, symmetric ideals and rings of Teter type are always understood to be in the multigraded sense. Precise definitions will be given below.

The above introduced concepts and the related statements can be also expressed in terms of poset language.
A {\em poset} $(\X, \preceq)$ is a set $\X$ together with a relation $\preceq$, which is reflexive, antisymmetric and transitive. The canonical monomial $K$-basis $\P$ of a monomial $K$-algebra $R$ has the structure of a poset: $u_1\preceq_{\P} u_2$ if $u_2|u_1$, or, equivalently, $\deg(u_2)\le \deg(u_1)$.  The dual basis $\P^*$, which is the monomial basis of $\omega_R$, has a dual poset structure: $u_2^*\preceq_{\P^*} u_1^*$ if and only if $\deg(u_1^*)\le \deg(u_2^*)$. Note that
$$
u_1\preceq_{\P} u_2\Leftrightarrow \deg(u_2)\le \deg(u_1)\Leftrightarrow \deg(u_1^*)\le \deg(u_2^*)\Leftrightarrow u_2^*\preceq_{\P^*} u_1^*.
$$

The poset $(\P,\preceq_{\P})$ will be called the {\em divisor poset} of $R$. The poset $(\P^*,\preceq_{\P^*})$ will be  called the divisor poset of $\omega_R$. By abuse of notation, we will denote these posets simply by $\P$ and $\P^*$, respectively. We will also write $\preceq$ for both $\preceq_{\P}$ and $\preceq_{\P^*}$ since there is no risk of confusion.
Note that since $R$ is $0$-dimensional, both $\P$ and $\P^*$ are finite posets.

Recall that a {\em poset ideal} of $(\X,\preceq)$  is a subset $\U \subseteq \X$ such that, if $u \in \U, v \in \X$ together with $v \preceq u$, then $v \in \U$. In other words, a poset ideal of $(\X,\preceq)$ is a downward closed subset of $\X$. Dually, an {\em order ideal} of $(\X,\preceq)$ is an upward closed subset of $\X$. Note that if $\I$ is a poset (respectively, order) ideal of $\P$, then $\I^*$ is an order (respectively, poset) ideal of $\P^*$. Conversely, every order (respectively, poset) ideal of $\P^*$ is of the form $\I^*$, where $\I$ is a poset (respectively, order) ideal of $\P$. Note that monomial ideals of $R$ are in bijection with poset ideals of $\P$ and therefore we will be only interested in poset ideals of $\P$ and order ideals of $\P^*$.

For a poset ideal $\I$ of $\P$, the set of maximal (respectively, minimal) elements of $\I$ will be denoted by $\Gen (\I)$ (respectively, $\Soc(\I)$). The same notations will be used for order ideals of $\P^*$. Note that $\Gen(\I^*)=(\Soc(\I))^*$ and $\Soc(\I^*)=(\Gen(\I))^*$.

Let $(\X,\preceq)$ be a poset. Then $u_1$ is called a {\em lower neighbor} of $u_2$, denoted by $u_1\precdot u_2$ (and $u_2$ is called an {\em upper neighbor} of $u_1$) if $u_1\prec u_2$ (that is, $u_1\preceq u_2$ and $u_1\neq u_2$), and there is no $v\in\X$ such that $u_1\prec v\prec u_2$. In this case one sometimes says that $u_2$ {\em covers} $u_1$. Note that $u_1\precdot u_2$ in $\P$ if and only if $u_2^*\precdot u_1^*$ in $\P^*$ if and only if $u_1/u_2=x_i$ for some $i$.




Let $\I$ be a poset ideal of $\P$ and $\J^*$ be an order ideal of $\P^*$. A bijection $\varphi : \J^* \to \I$ is called a {\em multigraded isomorphism} if for all $v_1^*,v_2^* \in \J^*$ we have

$$\deg(\varphi(v_1^*)) - \deg(v_1^*) = \deg(\varphi(v_2^*)) - \deg(v_2^*).$$
The equation above can be rewritten as $\varphi(v_1^*)v_1=\varphi(v_2^*)v_2$.
In this case we write $\J^*\iso \I$.
We define the {\em multidegree} of a multigraded isomorphism $\varphi$ to be $\deg(\varphi):=\deg(\varphi(v_1^*)) - \deg(v_1^*)$. In this section we will say   “isomorphism” to mean “multigraded isomorphism”.
\begin{Remark}
\label{dual map}
Clearly, if $\varphi: \J^* \to \I$ defines an isomorphism, then $\varphi^*: \I^* \to \J$ given by $\varphi^*(u^*)=(\varphi^{-1}(u))^*$ also defines an isomorphism. In other words, $\varphi^*(u^*)=v^*$ if and only if $\varphi(v)=u$. Moreover, $\deg(\varphi)=\deg(\varphi^*)$.
Indeed, if $u=\varphi(v)$, we get
\begin{multline*}
\deg(\varphi^*(u^*))-\deg(u^*)=\deg((\varphi^{-1}(u))^*)-\deg(u^*)=\\=\deg(v^*)-\deg(u^*)=\deg(u)-\deg(v)=\deg(\varphi(v))-\deg(v).
\end{multline*}

Note that any isomorphism $\varphi$ preserves and reflects covering relations. Indeed:
\begin{multline*}
v_1^* \precdot v_2^* \Leftrightarrow\ \exists i: v_2=v_1x_i \Leftrightarrow \exists i:\varphi(v_1^*)=\varphi(v_2^*)x_i \Leftrightarrow \varphi(v_1^*) \precdot \varphi(v_2^*).
\end{multline*}

\end{Remark}

A poset ideal $\I$ of $\P$ is called {\em symmetric} if $\I\iso \I^*$.

We say that a poset ideal $\J\subseteq \P$ is a {\em companion} of a poset ideal $\I\subseteq \P$ if $\I\iso \J^*$. In this case we also have $\J\iso\I^*$ and thus $\I$ is a companion of $\J$. We say that a poset ideal $\I \subseteq \P$ is a {\em $\tau$-ideal} if it has a companion.
Note that a poset ideal $\I\subseteq \P$ is symmetric if and only if it is a companion of itself. If these equivalent conditions hold, then $\I$ is clearly a $\tau$-ideal, but the converse is not true in general.

We remark that the notions of a companion, a $\tau$-ideal and a symmetric ideal defined here are the multigraded versions of those defined in \Cref{abstract}, translated into the poset language.
\begin{Example}
\label{homomorphism}

Let $R = K[x, y]/(x^3, y^4, xy^2)$.  The divisor posets $\P$ of $R$ and $\P^*$ of $\omega_R$ are drawn below.

\begin{figure}[hbt]
\begin{center}
\psset{unit=0.7cm}
\begin{pspicture}(3.62,5.51)(19.22,12.05)
\psline(8.,11.)(7.,10.)
\psline(8.,11.)(9.,10.)
\psline(9.,10.)(8.,9.)
\psline(7.,10.)(8.,9.)
\psline(7.,10.)(6.,9.)
\psline(8.,9.)(7.,8.)
\psline(6.,9.)(7.,8.)
\psline(9.,10.)(10.,9.)
\psline(10.,9.)(11.,8.)
\psline(15.,11.)(14.,10.)
\psline(15.,11.)(14.,10.)
\psline(14.,10.)(15.,9.)
\psline(15.,11.)(16.,10.)
\psline(16.,10.)(15.,9.)
\psline(16.,10.)(17.,9.)
\psline(15.,9.)(16.,8.)
\psline(17.,9.)(16.,8.)
\psline(18.,10.)(17.,9.)
\psline(18.,10.)(19.,11.)

\psdot(8.,11.)
\rput(7.95,11.40) {$1$}
\psdot(7.,10.)
\rput (6.64, 10.14) {$x$}
\psdot(9.,10.)
\rput (9.36, 10.1) {$y$}
\psdot(8.,9.)
\rput (8.55,8.88) {$xy$}
\psdot(6.,9.)
\rput (5.48,9) {$x^2$}
\psdots(7.,8.)
\rput (6.95,7.5) {$x^2y$}
\psdot(10.,9.)
\rput (10.54,9.1) {$y^2$}
\psdot(11.,8.)
\rput (11.44,8) {$y^3$}
\psdot(15.,11.)
\rput (15.05,11.52) {$(x^2y)^*$}
\psdot(14.,10.)
\rput (13.14,10) {$(x^2)^*$}
\psdot(16.,10.)
\rput (16.89, 10.18) {$(xy)^*$}
\psdot(15.,9.)
\rput (14.48, 8.78) {$x^*$}
\psdot(16.,8.)
\rput (16.10,7.4) {$1^*$}
\psdot(17.,9.)
\rput (17.5, 8.8) {$y^*$}
\psdot(18.,10.)
\rput (18.85, 9.96) {$(y^2)^*$}
\psdot(19.,11.)
\rput (19.78, 11) {$(y^3)^*$}

\end{pspicture}
\end{center}
\caption{The divisor posets $\P$ and $\P^*$}
\label{poset of Example 3.1}
\end{figure}
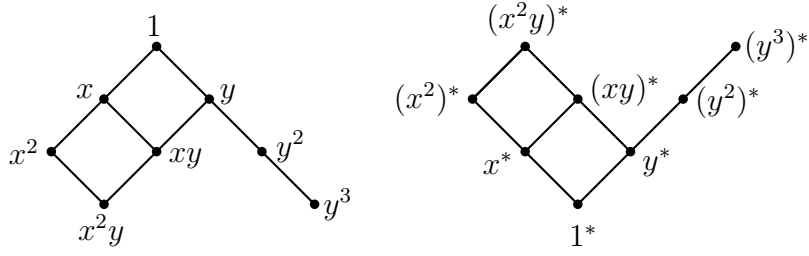

Let $\I = \{x, x^2, xy, x^2y\}$. Then $\I^* = \{(x^2y)^*, (x^2)^*, (xy)^*, x^*\}$.  Note that $\I$ is a poset ideal of $\P$ and thus $\I^*$ is an order ideal of $\P^*$.  The bijection $\varphi : \I^* \to \I$ defined by setting
\[
\varphi((x^2y)^*) = x, \, \varphi((x^2)^*) = xy, \, \varphi((xy)^*) = x^2, \, \varphi(x^*) = x^2y
\]
is an isomorphism. Shortly put, the isomorphism is given by $\varphi(v^*)=x^3y/v$ and the multidegree of this map is $(3,1)$. Therefore, $\I$ is a companion of itself, or, in other words, $\I$ is symmetric.

Let $\J = \{y^2, y^3\}$. Then $\J^*= \{(y^3)^*, (y^2)^*\}$.  The bijection $\psi : \J^* \to \J$ defined by setting
\[
\psi((y^3)^*) = y^2, \, \psi((y^2)^*) = y^3
\]
is an isomorphism. Shortly put, the isomorphism is given by $\psi(v^*)=y^5/v$ and the multidegree of this map is $(0,5)$.
Therefore, $\J$ is a companion of itself, or, in other words, $\J$ is symmetric.

Furthermore, the bijection $\mu : \I^*\cup \J^* \to \I \cup \J$ defined by setting
\[
\mu(u^*) =
\begin{cases}
\varphi(u^*) & \text{if  $u^* \in \I^*$},  \\
\psi(u^*) & \text{if  $u^* \in \J^*$}
\end{cases}
\]
is {\em not} an isomorphism. Moreover, it can be shown that there exists no isomorphism $\I^*\cup \J^*\to \I\cup \J$, even though the corresponding posets are isomorphic in the classical sense.

\end{Example}


In \Cref{homomorphism}, both $\I$ and $\J$ are $\tau$-ideals.  However, it can be shown that the poset ideal $\I \cup \J$ is {\em not} a $\tau$-ideal.

\begin{Example}
\label{EXAMPLE_Z}

Let $R = K[x,y]/(x^5, y^4, x^2y^2, x^4y)$.  The divisor posets $\P$ of $R$ and $\P^*$ of $\omega_R$ are drawn below.

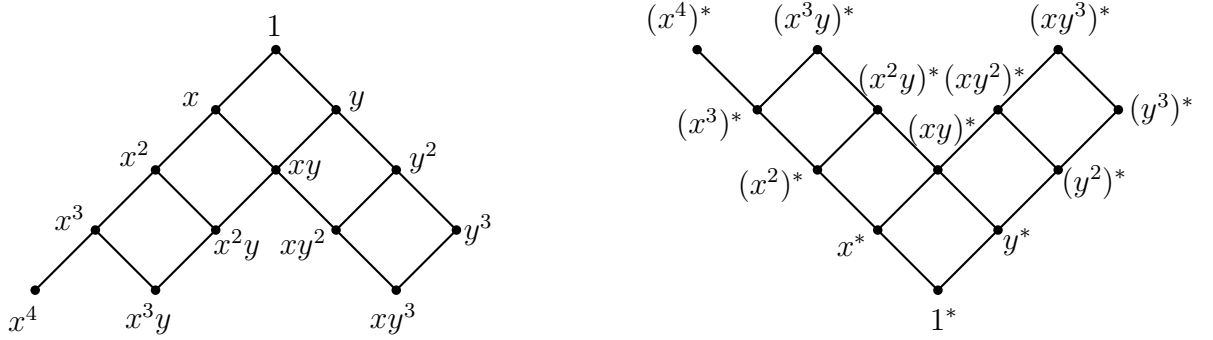
\begin{figure}[hbt]
\begin{center}
\psset{unit=0.8cm}

\begin{pspicture}(3.92,9.1)(20.76,16.92)
\psline(9.,15.)(8.,14.)
\psline(8.,14.)(7.,13.)
\psline(7.,13.)(6.,12.)
\psline(6.,12.)(5.,11.)
\psline(9.,15.)(10.,14.)
\psline(10.,14.)(11.,13.)
\psline(11.,13.)(12.,12.)
\psline(8.,14.)(9.,13.)
\psline(10.,14.)(9.,13.)
\psline(9.,13.)(8.,12.)
\psline(7.,13.)(8.,12.)
\psline(9.,13.)(10.,12.)
\psline(11.,13.)(10.,12.)
\psline(8.,12.)(7.,11.)
\psline(6.,12.)(7.,11.)
\psline(10.,12.)(11.,11.)
\psline(12.,12.)(11.,11.)
\psline(20.,11.)(21.,12.)
\psline(20.,11.)(19.,12.)
\psline(21.,12.)(20.,13.)
\psline(19.,12.)(20.,13.)
\psline(21.,12.)(22.,13.)
\psline(20.,13.)(21.,14.)
\psline(21.,14.)(22.,13.)
\psline(23.,14.)(22.,13.)
\psline(22.,15.)(21.,14.)
\psline(22.,15.)(23.,14.)
\psline(19.,12.)(18.,13.)
\psline(20.,13.)(19.,14.)
\psline(19.,14.)(18.,13.)
\psline(17.,14.)(18.,13.)
\psline(18.,15.)(19.,14.)
\psline(18.,15.)(17.,14.)
\psline(16.,15.)(17.,14.)

\psdot(9.,15.)
\rput(8.97,15.40) {$1$}
\psdot(8.,14.)
\rput(7.59,14.18) {$x$}
\psdot(10.,14.)
\rput(10.35,14.15) {$y$}
\psdot(9.,13.)
\rput(9.48,13) {$xy$}
\psdot(7.,13.)
\rput(6.65,13.27) {$x^2$}
\psdot(11.,13.)
\rput(11.45,13.20) {$y^2$}
\psdot(8.,12.)
\rput(8.35,11.75) {$x^2y$}
\psdot(10.,12.)
\rput(9.45,11.75) {$xy^2$}
\psdot(6.,12.)
\rput(5.57,12.2) {$x^3$}
\psdot(12.,12.)
\rput(12.36,12) {$y^3$}
\psdot(5.,11.)
\rput(4.8,10.52) {$x^4$}
\psdot(7.,11.)
\rput(6.88,10.52) {$x^3y$}
\psdot(11.,11.)
\rput(10.95,10.52) {$xy^3$}
\psdots(20.,11.)
\rput(20.1,10.52) {$1^*$}
\psdots(19.,12.)
\rput(18.57,11.75) {$x^*$}
\psdot(18.,13.)
\rput(17.23,12.75) {$(x^2)^*$}
\psdot(21.,12.)
\rput(21.32,11.80) {$y^*$}
\psdot(20.,13.)
\rput(20.07,13.65) {$(xy)^*$}
\psdot(22.,13.)
\rput(22.6,12.80) {$(y^2)^*$}
\psdot(17.,14.)
\rput(16.21,13.75) {$(x^3)^*$}
\psdot(19.,14.)
\rput(19.35,14.5) {$(x^2y)^*$}
\psdot(21.,14.)
\rput(20.78, 14.5) {$(xy^2)^*$}
\psdot(23.,14.)
\rput(23.72,14) {$(y^3)^*$}
\psdot(16.,15.)
\rput(15.70,15.50) {$(x^4)^*$}
\psdot(18.,15.)
\rput(17.94,15.50) {$(x^3y)^*$}
\psdot(22.,15.)
\rput(22.2,15.50) {$(xy^3)^*$}
\end{pspicture}
\end{center}
\caption{The divisor posets $\P$ and $\P^*$}
\label{Example 4.2}
\end{figure}

\noindent
Let $\I = \{x^3, x^4, x^3y\}$ and $\J = \{y^3, xy^2, xy^3\}$ be poset ideals of $\P$.  Then $\I^* = \{(x^4)^*, (x^3y)^*, (x^3)^*\}$ and $\J^* = \{(xy^3)^*, (y^3)^*, (xy^2)^*\}$.  Then the bijection $\varphi : \I^* \cup \J^* \to \I \cup \J$ defined by setting
\[
\varphi((xy^3)^*) = x^3, \, \, \, \varphi((y^3)^*) = x^4, \, \, \, \varphi((xy^2)^*) = x^3y,
\]
\[
\varphi((x^4)^*) = y^3, \, \, \, \varphi((x^3y)^*) = xy^2, \, \, \, \varphi((x^3)^*) = xy^3
\]
is an isomorphism. Shortly put, the isomorphism is given by $\varphi(v^*)=x^4y^3/v$ and the multidegree of this map is $(4,3)$.
Thus each of the poset ideals $\I, \J$ (by restricting the map accordingly) and $\I \cup \J$ is a $\tau$-ideal. Note that $\I$ and $\J$ are companions via a restriction of this map. Furthermore, $\I \cup \J$ is symmetric via $\varphi$.

Note that the $\tau$-ideal $\I$ is not even symmetric in the local sense. Indeed, if $\I$ is symmetric, then by Remark \ref{Juergen}, for the corresponding monomial ideal $I$, the number of generators $\mu(I)$ of $I$ coincides with the socle dimension $\sigma(I)$ of $I$. This is not the case here, since $\mu(I)=|\Gen(\I)|=1$ and $\sigma(I)=|\Soc(\I)|=2$.

\end{Example}

\begin{Lemma}
\label{MIGHT_BE_TRUE}
Let $\I$ and $\J$ be companions.  Then $\I \cup \J$ is symmetric.
\end{Lemma}
\begin{proof}
Let $\I$ and $\J$ be companions. Then there is an isomorphism $\varphi: \J^* \to \I$. By \Cref{dual map}, there is a dual map $\varphi^*: \I^* \to \J$ with $\deg(\varphi)=\deg(\varphi^*)$. Then we define a map $\psi: \I^*\cup \J^* \to \I\cup \J$ by setting

\[
\psi(u^*)=
\begin{cases}
\varphi(u^*) & \text{if  $u^* \in \J^*$},  \\
\varphi^*(u^*) & \text{if  $u^* \in \I^*$}.
\end{cases}
\]
Since $\deg(\varphi)=\deg(\varphi^*)$, these maps agree on $I^*\cap J^*$ and therefore $\psi$ is well defined. Moreover, $\psi$ is an isomorphism and $\deg(\psi)=\deg(\varphi)=\deg(\varphi^*)$.
\end{proof}
The {\em trace} of $\P^*$, denoted by $\tr(\P^*)$, is defined to be the union of all $\tau$-ideals of $\P$. We say that $\P$ is {\em of Teter type} if $\tr(\P^*)$ is a $\tau$-ideal.  In other words, $\P$ is of Teter type if there is an order ideal $\J^*$ of $\P^*$ and an isomorphism $\varphi: \J^* \to \tr(\P^*)$.

As an immediate corollary of Lemma \ref{MIGHT_BE_TRUE} we have

\begin{Corollary}(Compare with \Cref{easy})\label{unionsymmetric}
The trace of $\P^*$ is the union of the symmetric ideals of $\P$. In particular, $\P$ is of Teter type if and only if $\tr(\P^*)$ is a symmetric poset ideal of $\P$
\end{Corollary}

\begin{Example}
\label{intersection_of_I_and_J}

Let $R = K[x,y]/(x^6, x^4y^2, x^2y^4, xy^5, y^6)$.  The divisor poset $\P$ is drawn below.

\begin{figure}[hbt]
\begin{center}
\psset{unit=0.8cm}
\begin{pspicture}(3.,5.)(19.,13.5)

\psline (11.,13.)(10.,12.)
\psline (11.,13.)(12.,12.)
\psline (10.,12.)(11.,11.)
\psline(12.,12.)(11.,11.)
\psline(10.,12.)(9.,11.)
\psline(10.,10.)(11.,11.)
\psline(10.,10.)(9.,11.)
\psline(12.,12.)(13.,11.)
\psline(11.,11.)(12.,10.)
\psline(13.,11.)(12.,10.)
\psline(9.,11.)(8.,10.)
\psline(10.,10.)(9.,9.)
\psline(8.,10.)(9.,9.)
\psline(8.,10.)(7.,9.)
\psline(9.,9.)(8.,8.)
\psline(7.,9.)(8.,8.)
\psline(7.,9.)(6.,8.)
\psline(8.,8.)(7.,7.)
\psline(6.,8.)(7.,7.)
\psline(10.,10.)(11.,9.)
\psline(12.,10.)(11.,9.)
\psline(9.,9.)(10.,8.)
\psline(11.,9.)(10.,8.)
\psline(13.,11.)(14.,10.)
\psline(12.,10.)(13.,9.)
\psline(11.,9.)(12.,8.)
\psline(14.,10.)(13.,9.)
\psline(13.,9.)(12.,8.)
\psline(14.,10.)(15.,9.)
\psline(13.,9.)(14.,8.)
\psline(15.,9.)(14.,8.)
\psline(12.,8.)(11.,7.)
\psline(10.,8.)(11.,7.)
\psline(15.,9.)(16.,8.)

\psdot(11.,13.)
\psdot(10.,12.)
\psdot(12.,12.)
\psdot(9.,11.)
\psdot(11.,11.)
\psdot(13.,11.)
\psdot(8.,10.)
\psdot(10.,10.)
\psdot(12.,10.)
\psdot(14.,10.)
\psdot(7.,9.)
\psdot(9.,9.)
\psdot(11.,9.)
\psdot(13.,9.)
\psdot(15.,9.)
\psdot(6.,8.)
\psdot(8.,8.)
\psdot(10.,8.)
\psdot(12.,8.)
\psdot(14.,8.)
\psdot(16.,8.)
\psdot(7.,7.)
\psdot(11.,7.)

\rput(11.0,13.45){$1$}
\rput(9.71,12.28){$x$}
\rput(8.68,11.25){$x^2$}
\rput(7.7,10.25){$x^3$}
\rput(6.8,9.3){$x^4$}
\rput(5.59,7.98){$x^5$}
\rput(6.98,6.55){$x^5y$}
\rput(8.22,7.67){$x^4y$}
\rput(9.51,7.64){$x^3y^2$}
\rput(10.98,6.53){$x^3y^3$}
\rput(12.4,7.72){$x^2y^3$}
\rput(14.1,7.59){$xy^4$}
\rput(16.4,7.96){$y^5$}
\rput(15.42,9.28){$y^4$}
\rput(14.46,10.29){$y^3$}
\rput(13.45,11.25){$y^2$}
\rput(12.35,12.22){$y$}
\rput(10.95,8.33){$x^2y^2$}
\end{pspicture}
\end{center}
\caption{The divisor poset $\P$}
\label{poset of Example 4.5}
\end{figure}
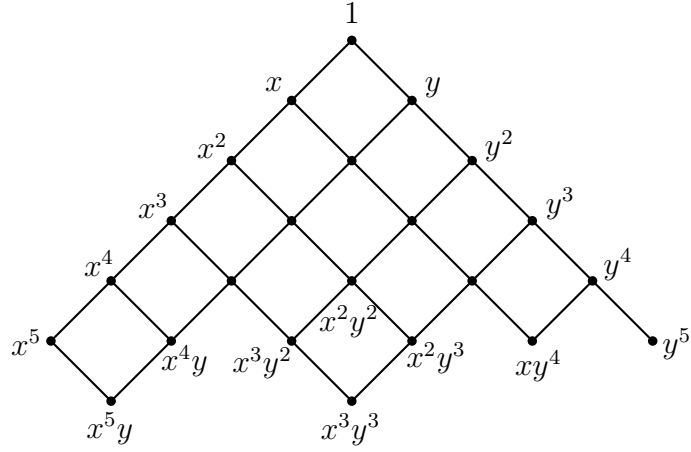

\noindent
Let $\I_1 = \{x^5, x^5y, x^4y\}, \I_2 = \{x^2y^2, x^3y^2, x^2y^3, x^3y^3\},  \I_3 = \{y^4, xy^4, y^5\}$ be poset ideals of $\P$. Let $\I = \I_1 \cup \I_2$ and $\J = \I_2 \cup \I_3$.  It then follows that $\J$ is a companion of $\I$ via the corresponding map $\J^*\to \I$ of multidegree $(5,5)$.  The poset ideal $\I \cup \J$ is symmetric.

\end{Example}



We will now address the following question: for given poset ideals $\I, \J\subseteq \P$, how can one determine whether $\I$ and $\J$ are companions? The answer is given in the next lemma:

\begin{Lemma}
\label{bipartite}
Let $\I$ and $\J$ be poset ideals with $\Gen(\I)=\{a_1,\ldots, a_{k}\}$, $\Soc(\I)=\{b_1,\ldots, b_{l}\}$, $\Gen(\J)=\{c_1,\ldots, c_{l'}\}$, $\Soc(\J)=\{d_1,\ldots, d_{k'}\}$. Then $\I$ and $\J$ are companions if and only if  $k=k'$, $l=l'$ and after a relabelling of elements of $\Gen(\J)$ and $\Soc(\J)$ we have $a_id_i=b_jc_j$ for all $1\le i\le k$ and $1\le j\le l$
\end{Lemma}
\begin{proof}
$\Rightarrow:$ Let $\varphi: \J^* \to \I$ be an isomorphism and let $m$ be the monomial with $\deg(m)=\deg(\varphi)$. Then $\varphi$ is given by $\varphi(v^*)=m/v$ for all $v\in \J$. In particular, this map defines a bijection $\Gen(\J^*)\to\Gen (\I)$. Note that
$\Gen(\J^*)=(\Soc(\J))^*=\{d^*_{1},\ldots, d^*_{k'}\}$. This implies $k=k'$ and, after a suitable relabelling, for all $1\le i\le k$ we have $\varphi(d^*_i)=m/d_i=a_i$, that is to say, $a_id_i=m$. A similar argument can be applied to the socles to conclude $l=l'$ and $b_jc_j=m$ for all $1\le j\le l$.

$\Leftarrow:$ Let $k=k'$, $l=l'$ and let $a_id_i=b_jc_j=m$ for $1\le i\le k$ and $1\le j\le l$. We will show that the map $\varphi$ with $\deg(\varphi)=\deg(m)$ defines an isomorphism $\J^*\to \I$. We have $v^*\in \J^*$ if and only if there exist a socle of $c^*_j$ of $\J^*$ and a generator $d^*_i$ of $\J^*$ such that $c^*_j\preceq v^*\preceq d^*_i$, or, in other words, $c_j|v$ and $v|d_i$. Then $v|m$ and
$$(m/d_i)|(m/v) \text{ and } (m/v)|(m/c_j)\Leftrightarrow a_i|(m/v) \text{ and } (m/v)|b_j\Leftrightarrow b_j \preceq m/v\preceq a_i,$$
which is equivalent to saying that $m/v\in \I$. This proves that the map is well defined on all of $\J^*$. Clearly, $\varphi$ is injective. For surjectivity, one applies the argument above in the opposite direction: if $u\in \I$, then $u|m$ and $v^*:=(m/u)^*\in \J^*$ with $\varphi(v^*)=m/v=u$.
\end{proof}
\begin{Corollary}
\label{symmetric}
Let $\I$ be a poset ideal with $\Gen(\I)=\{a_1,\ldots, a_k\}$ and $\Soc(\I)=\{b_1,\ldots, b_l\}$. Then $\I$ is symmetric if and only if $k=l$ and after a relabelling of elements of $\Soc(\I)$ we have $a_ib_i=a_jb_j$ for all $1\le i,j\le k$.
\end{Corollary}

\begin{Remark}
Assume that the equivalent conditions in \Cref{bipartite} hold. Then there exists a {\em unique} $m$ such that after a suitable relabelling, $a_id_i=b_jc_j=m$ for all $1\le i\le k$ and $1\le j\le l$. Indeed, having such a relabelling in particular implies $a_1\cdots a_kd_1\cdots d_k=m^k$ and thus $m$, given that it exists, can be uniquely determined. Moreover, there is a {\em unique} isomorphism $\varphi: \J^*\to \I$, namely, the one given by $\varphi(v^*)=m/v$.

\end{Remark}

\begin{Example}
\label{TRACE_NOT_CONNECTED_TETER_TYPE}


Let $I=(x^4,y^4,x^2y^2)$, $J\subseteq I$ such that $G=K[x,y]/J$ is a $0$-dimensional local $K$-algebra. Let  $\epsilon: G\to R=K[x,y]/I$ be the canonical epimorphism.
We claim that there is no $k$ such that $\Ker(\epsilon)=0:\mm_G^k$, where $\mm_G$ is the maximal ideal of $G$. Indeed, suppose $\Ker(\epsilon)=0:\mm_G^k$. Then $I=J:\mm^k$, where $\mm=(x,y)$. Then $(x^4)\mm^k\subseteq J$ and $(x^2y^2)\mm^k\subseteq J$. We will show that $(x^3y)\mm^k\subseteq J$, giving a contradiction since $x^3y\not\in I$. Take $x^{k_1}y^{k-k_1}\in \mm^k$. Then:
\begin{itemize}
\item if $k_1\ge 1$, then $x^3y\cdot x^{k_1}y^{k-k_1}=x^4(x^{k_1-1}y^{k-k_1+1})\in (x^4)\mm^k\subseteq J$;
\item if $k_1=0$, then $x^3y\cdot y^{k}=x^2y^2(xy^{k-1})\in (x^4)\mm^k\subseteq J$.
\end{itemize}
On the other hand, we will prove that $R$ is of Teter type.

\noindent
Let $\I$ denote the poset ideal consisting of
\[
x^2, \, x^3, \, x^2y, \, x^3y, \, y^2, \, y^3, \, xy^2, \, xy^3.
\]
From \Cref{tracemaci} we know that $\tr(\P^*) = \I$.

We have $\Gen(\I)=\{x^2,y^2\}$ and $\Soc(\I)=\{x^3y,xy^3\}$. We have $x^2\cdot y^2\cdot x^3y\cdot xy^3=x^6y^6$ and thus the isomorphism $\I^* \to \I$, if it exists, should have multidegree $(3,3)$. After a relabelling of $\Soc \I$ we indeed have $x^2\cdot xy^3=y^2\cdot x^3y=x^3y^3$. Hence $\I$ is symmetric. Therefore, $\P$ is of Teter type.

\end{Example}

We now turn to the study of $0$-dimensional monomial $K$-algebras of Teter type.  Let $S = K[x_1, \ldots, x_n]$ denote the polynomial ring and
\[
R = S/(x_1^{a_1+1}, x_2^{a_2+1}, \ldots, x_n^{a_n+1}),
\]
where $1 \leq a_1 \leq a_2 \leq \cdots \leq a_n$, and $\nn$ is the graded maximal ideal of $R$.

\begin{Theorem}
\label{Chopin}
Let $1 \leq k \leq a_1$.  Then the divisor poset $\P$ of $R/(0:\nn^k)$ is of Teter type.  Furthermore, the trace of the divisor poset $\P^*$ is $\I_k$, where $\I_k$ is the poset ideal of $\P$ whose $\Gen(\I_k)$ consists of those monomials $x_1^{b_1}x_2^{b_2}\ldots x_n^{b_n}$ with each $0 \leq b_i \leq a_i$ and with $\sum_{i=1}^n b_i = k$.
\end{Theorem}

\begin{proof}
One claims that the ideal $\I_k \subseteq \P$ is symmetric.  Since $\Soc(\I_k)$ consists of those monomials $x_1^{a_1 - b_1}x_2^{a_2 - b_2}\ldots x_n^{a_n - b_n}$ with $\sum_{i=1}^n b_i = k$, it follows from \Cref{symmetric} that $\I_k \subseteq \P$ is symmetric.  Hence $\I_k \subseteq \tr(\P^*)$.

Now, it is shown that every $\tau$-ideal of $\P$ is contained in $\I_k$.  Let $\I$ be a $\tau$-ideal of $\P$ with $\I \nsubseteq \I_k$.  Let $u$ be a monomial of $\Gen(\I)$ of the smallest degree.  Then $d = \deg u < a_1$.  Let $\J$ be a companion of $\I$.  Since each monomial belonging to $\Soc(\J)$ is of degree $\sum_{i=1}^{n} a_i - k$, it follows that each monomial of $\Gen(\I)$ is of degree $d$.  Since $d < a_1$, each monomial of $\Gen(\I)$ possesses exactly $n$ lower neighbors.  It then follows that each monomial belonging to $\Soc(\J)$ possesses exactly $n$ upper neighbors in $\J$.  In other words, if $u = x_1^{c_1}x_2^{c_2}\ldots x_n^{c_n} \in \Soc(\J)$, then $c_i > 0$ for all $i$ and $u/x_i \in \J$.  Let $<_{\rm lex}$ denote the lexicographic order induced by $x_1 < x_2 < \cdots < x_n$.  Let $w = x_1^{c_1}x_2^{c_2}\ldots x_n^{c_n}$ be the monomial belonging to $\Soc(\J)$ which is the largest with respect to $<_{\rm lex}$.  Let $1 \leq i_0 \leq n$ be the largest integer for which $c_{i_0} < a_{i_0}$.  If $i_0 > 1$ and $w_0 = x_{i_0}(w/x_{i_0-1})$, then $w <_{\rm lex} w_0$ and $w_0 \in \J$, a contradiction.  Hence $i_0 = 1$ and $w = x_1^{a_1-k}x_2^{a_2}\ldots x_n^{a_n}$.  Then $x_1^{a_1 - k + \sum_{i=2}^{n} e_i}(w/x_2^{e_2} \cdots x_n^{e_n}) \in \J$, where $\sum_{i=2}^{n} e_i \leq k$.  Hence $\Soc(\J) = \Soc(\I_k)$.  However, since $|\Gen(\J)| < |\Soc(\I_k)|$, it follows that $\J$ cannot be a companion of $\I$, as desired.
\end{proof}

One can conjecture that, if an integer $k \geq 1$ satisfies
\[
|\Gen(\I_1)| < |\Gen(\I_2)| < \cdots < |\Gen(\I_{k})|,
\]
then $\P$ is of Teter type and the trace of $\P^*$ is $\I_k$.

\begin{Example}
\label{k=a+1}
{\rm
Let $n \geq 3$ and $a_1 = \cdots =a_n = a$.  Let $k = a + 1$.  In the divisor poset $\P$ of $R/(0:\nn^k)$, the poset ideal $\I_k$ is symmetric.  Hence $\I_k \subset \tr(\P^*)$.  One claims every $\tau$-ideal of $\P$ is contained in $\I_k$.  Let $\I$ be a $\tau$-ideal of $\P$ with $\I \nsubseteq \I_k$ and $\J$ a companion of $\I$.  Let $u$ be a monomial of $\Gen(\I)$ of the smallest degree.  Then $d = \deg u \leq a$.  When $d < a$, the proof of \Cref{Chopin} can be available.  Let $d = a$.  Since each monomial belonging to $\Soc(\J)$ is of degree $(n - 1)a - 1$, it follows that each monomial of $\Gen(\I)$ is of degree $a$.  Thus in particular, except for the monomials $x_1^a, \ldots, x_n^a$, each monomial belonging to $\Gen(\I)$ possesses exactly $n$ lower neighbors.  Let $<_{\rm lex}$ denote the lexicographic order induced by $x_1 < x_2 < \cdots < x_n$.  Let $w = x_1^{c_1}x_2^{c_2}\ldots x_n^{c_n}$ be the monomial belonging to $\Soc(\J)$ which is the largest with respect to $<_{\rm lex}$.  Let $c_n < a$.  Since $w$ possesses at least $n - 1$ upper neighbors, there is $1 \leq i < n$ for which $x_i$ divides $w$.  Since $\J$ is a poset ideal, one has $w' = x_n(w/x_i) \in \J$.  However, $w <_{\rm lex} w'$, a contradiction.  Thus $c_n = a$.  Similarly, one has $c_3 = \cdots = c_n = a$.  Let $w = x_1^{c_1} x_2^{c_2} x_3^a \cdots x_n^a$.  If $w/x_1 \in \J$, then $w'' = x_2(w/x_1) \in \J$ and $w <_{\rm lex} w''$, a contradiction.  Hence the monomial $w_0$ of $\I$ which is accompanied by $w \in \J$ is $x_1^a$.  Since $c_1 + c_2 = a - 1$, the lower neighbors of $w/x_n$ are $x_1(w/x_n), x_2(w/x_n)$ and $w$.  However, the upper neighbors of $x_n w_0 = x_1^a x_n$ are $x_1^a$ and $x_1^{a-1}x_n$, a contradiction.  Hence $\I$ cannot be a $\tau$-ideal, as desired.  Thus $\I_k = \tr(\P^*)$.
}
\end{Example}

\begin{Theorem}
\label{Mozart}
Let $a_1 = a_2 = \cdots = a_n = 1$ and $1 \leq k \leq n - 1$.  Let $\P$ denote the divisor poset of $R/(0:\nn^k)$.  Then the trace of $\P^*$ is $\I_{k_0}$, where $k_0 = \min\{k, n - k\}$ and where $\I_{k_0}$ is the poset ideal of $\P$ whose $\Gen(\I_{k_0})$ consists of those monomials $x_1^{b_1}x_2^{b_2}\ldots x_n^{b_n}$ with each $0 \leq b_i \leq 1$ and with $\sum_{i=1}^n b_i = k_0$.  Furthermore, $\P$ is of Teter type if and only if $k \leq n - k$.
\end{Theorem}

\begin{proof}
If $k \leq n - k$, then $\I_{k_0} \subseteq \P$ is symmetric.  In fact, since $\Gen(\I_{k_0})$ consists of squarefree monomials in $x_1, \ldots, x_n$ of degree $k_0$ and $\Soc(\I_{k_0})$ does of those of degree $n - k_0$,  \Cref{symmetric} says that $\I_{k_0}$ is symmetric.  If $k > n - k$, then $\Gen(\I_{k_0}) = \Soc(\I_{k_0})$.  Thus, for each $1 \leq k \leq n - 1$, one has $\I_{k_0} \subseteq \tr(\P^*)$.

Let $\I$ be a $\tau$-ideal of $\P$ with $\I \nsubseteq \I_{k_0}$.  Let $u$ be a monomial of $\Gen(\I)$ of the smallest degree.  Then $d = \deg u < k_0$.  Let $\J$ be a companion of $\I$.  Since each monomial belonging to $\Soc(\J)$ is of degree $n - k$, it follows that each monomial of $\Gen(\I)$ is of degree $d$.  Hence each monomial of $\Gen(\I)$ possesses exactly $n - d$ lower neighbors.  It then follows that each monomial belonging to $\Soc(\J)$ possesses exactly $n - d$ upper neighbors in $\J$.  On the other hand, each monomial belonging to $\Soc(\I_{k_0})$ possesses exactly $n - k$ upper neighbors.  Hence $n - k \geq n - d$.  Thus $k \leq d < k_0$, which contradicts $k_0 = \min\{k, n - k\}$.  Hence every $\tau$-ideal is contained in $\I_{k_0}$.  Thus $\tr({\P^*}) = \I_{k_0}$, as required.  In particular, $\P$ is of Teter type if $k \leq n - k$.

Now, suppose that $k > n - k$.  Then $\tr(\P^*) = \I_{n-k}$, where  $\Gen(\I_{n-k})$ consists of squarefree monomials in $x_1, \dots, x_n$ of degree $n - k$.  One has $\Gen(\I_{n-k}) = \Soc(\I_{n-k})$.  If $\P$ is of Teter type, then one can find a bijection $\delta : \Gen(\I_{n-k}) \to \Gen(\I_{n-k})$ for which there is $c = (c_1, \ldots, c_n)\in \ZZ^n$ with $\deg(u) +  \deg(\delta(u)) = c$ for all $u \in \Gen(\I_{n-k})$.  In particular, $c_i \geq 1$ for all $i$.  However, since $\sum_{i=1}^{n} c_i = 2(n - k) < n$, it follows that $c_j = 0$ for some $1 \leq j \leq n$.  This contradiction shows that $\P$ cannot be of Teter type.
\end{proof}

\section{Monomial almost complete intersections of Teter type}

In this section we describe the trace of $\omega_R$ for an almost complete intersection monomial algebra $R$ and determine when such rings are  of Teter type.
\begin{Theorem}
\label{tracemaci}
Let $S=K[x_1,\ldots,x_n]$ be the polynomial ring over the field $K$ and $I=(x_1^{a_1},\ldots,x_n^{a_n},x_1^{b_1}\cdots x_n^{b_n})\subseteq S$ such that $b_i<a_i$ for all $i$ and $b_i>0$ for at least two integers $i$. Let $R=S/I$. Then
\[\tr(\omega_R)=(x_i^{a_i-b_i}:\ b_i>0)+(w/x_i^{b_i}:\ b_i>0),\]
where $w=x_1^{b_1}\cdots x_n^{b_n}$.
\end{Theorem}

\begin{proof}
Let $J=(x_1^{a_1},\ldots,x_n^{a_n})$, $L=J:w=(x_1^{a_1-b_1},\ldots,x_n^{a_n-b_n})$ and $\bb=(b_1,\ldots,b_n)$. Then we have the multigraded short exact sequence
\[
0\to (S/L)(-\bb)\stackrel{\varphi}{\To}S/J\to S/I\to 0,
\]
where the map $\varphi$ is multiplication by $w$. Let $\FF=K(x_1^{a_1},\ldots,x_n^{a_n};S)$ be the Koszul complex attached to the regular sequence $x_1^{a_1},\ldots,x_n^{a_n}$ which is indeed a minimal graded free $S$-resolution of $S/J$ and let $\GG=K(x_1^{a_1-b_1},\ldots,x_n^{a_n-b_n};S)$ which is a minimal graded free $S$-resolution of $S/L$. Then the map $\varphi$ can be lifted to a multigraded complex homomorphism $\{\varphi_i:\GG(-\bb)\to \FF\}$, where $\varphi_i(e_{j_1\cdots j_i})=(w/x_{j_1}^{b_{j_1}}\cdots x_{j_i}^{b_{j_i}})e'_{j_1 \cdots j_i}$ for any basis element $e_{j_1\cdots j_i}$ of $G_i(-\bb)$. So we have the commutative diagram

\bigskip
$\begin{CD}
\GG(-\bb):\ \ 0 @>>> G_n(-\bb) @>\partial_n>> G_{n-1}(-\bb) @>\partial_{n-1}>> G_{n-2}(-\bb) @>\partial_{n-2}>> \cdots \\
@. @V\cong V\varphi_n V @VV\varphi_{n-1}V  @VV\varphi_{n-2}V  \\
\FF:\ \ \ \ \ \ \ \ \ 0 @>>> F_n @>\partial'_n>> F_{n-1}@>\partial'_{n-1}>> F_{n-2}@>\partial'_{n-2}>> \cdots
\end{CD}
$

\bigskip

Let $\DD$ be the mapping cone of $\GG(-\bb)\to \FF$. Then $\DD$ is a graded free $S$-resolution of $R$ with the differential maps $d_i:F_i\dirsum G_{i-1}(-\bb)\to F_{i-1}\dirsum G_{i-2}(-\bb)$ defined as $$d_i(f,g)=(\varphi_{i-1}(g)+\partial'_i(f),-\partial_{i-1}(g))$$ for $1\leq i\leq n+1$. For any set $B=\{i_1,\ldots,i_s\}\subseteq [n]$, let $e_B=e_{i_1\cdots i_s}$ and $e'_B=e'_{i_1\cdots i_s}$. Then $\{e_{[n]\setminus i}:\ 1\leq i\leq n\}$ is a basis for $G_{n-1}(-\bb)$.
Let $$\overline{d}_n:G_{n-1}(-\bb)\to F_{n-1}\dirsum G_{n-2}(-\bb)$$ be the $S$-module homomorphism with $\overline{d}_n(e_{[n]\setminus i})=d_n(e_{[n]\setminus i})=(x_i^{b_i}e'_{[n]\setminus i},-\partial_{n-1}(e_{[n]\setminus i}))$ for any $1\leq i\leq n$.
Then clearly  $\overline{d}_n$ is an injective map. Also from the equalities
$$d_n(e'_{[n]})=\partial'_n(e'_{[n]})=\partial'_n \varphi_n(e_{[n]})=\varphi_{n-1} \partial_n(e_{[n]})=d_n\partial_n(e_{[n]})$$
we obtain $d_n(F_n)\subseteq d_n(G_{n-1}(-\bb))=\overline{d}_n(G_{n-1}(-\bb))$. This implies that $\Im \overline{d}_n=\Im d_n$. Therefore
$$\begin{CD}
\overline{\DD}:\ 0 \rightarrow G_{n-1}(-\bb) @>\overline{d}_n>> F_{n-1}\dirsum G_{n-2}(-\bb)@>d_{n-1}>> F_{n-2}\dirsum G_{n-3}(-\bb) @>d_{n-2}>>\cdots
\end{CD}$$
is a free $S$-resolution of $S/I$.

\bigskip

Consider the complex $\overline{\DD}\tensor_SR$ with the differential maps $\sigma_i$.  We have $\overline{D}_{n}=\bigoplus_{i=1}^n Se_{[n]\setminus i}$. Hence $\overline{D}_{n}\tensor_SR=\bigoplus_{i=1}^n Re_{[n]\setminus i}$ and
$\sigma_n(e_{[n]\setminus i})=(x_{i}^{b_i}e'_{[n]\setminus i},-\partial_{n-1}(e_{[n]\setminus i}))$ for $1 \leq i \leq n$. Here for simplicity the residue class of the monomial $x_i^{b_i}$ in $R$ is denoted again by $x_i^{b_i}$.
Let $C$ be the kernel of $\sigma_n:\overline{D}_n\tensor_SR\to \overline{D}_{n-1}\tensor_SR$. Then by \Cref{vasconcelos},
$\tr(\omega_R)$ is the ideal in $R$ generated by the entries of a generating set of $C$. For any $1\leq i\leq n$, we have $\sigma_n((w/x_i^{b_i})e_{[n]\setminus i})=0$. Indeed, for any $j\neq i$ the coefficient of $e_{[n]\setminus \{i,j\}}$ in $\partial_{n-1}(e_{[n]\setminus i})$ is $\pm x_{j}^{a_j-b_j}$ and $(w/x_i^{b_i})x_{j}^{a_j-b_j}=0_R$, since $(w/x_i^{b_i})x_{j}^{a_j-b_j}$ as a monomial in $S$ is divisible by $x_j^{a_i}$. So $(w/x_i^{b_i})e_{[n]\setminus i}\in C$ for any $1\leq i\leq n$. Now, consider the element $z=\partial_n(e_{[n]})=\sum_{i=1}^n (-1)^i x_i^{a_i-b_i} e_{[n]\setminus i}\in \overline{D}_{n}\tensor_SR$. We have $\sigma_n(z)=(\sum_{i=1}^n (-1)^i x_i^{a_i}e'_{[n]\setminus i},-\partial_{n-1}\partial_n(z))=0$, which implies that $z\in C$.
Without loss of generality assume that $b_i>0$ for $1\leq i\leq r$ and $b_{r+1}=\cdots=b_n=0$.
Then by \Cref{vasconcelos}, $$(x_i^{a_i-b_i}:\ 1\leq i\leq r)+(w/x_i^{b_i}:\ 1\leq i\leq r)\subseteq \tr(\omega_R).$$

Now, consider an arbitrary multigraded element $h=\sum_{i=1}^n s_i u_ie_{[n]\setminus i}\in C$, where $u_i$'s are monomials  and $s_i\in K$. Then $\sum_{i=1}^n s_i x_i^{b_i}u_ie_{[n]\setminus i}=0$ and so $x_i^{b_i}u_i\in I$ for any $1\leq i\leq n$. This means either $x_i^{a_i-b_i}|u_i$ or $(w/x_i^{b_i})|u_i$. Therefore $$\tr(\omega_R)\subseteq (x_i^{a_i-b_i}:\ 1\leq i\leq r)+(w/x_i^{b_i}:\ 1\leq i\leq r).$$ Thus the equality holds.
\end{proof}

The next result characterizes monomial almost complete intersection rings of Teter type.

\begin{Theorem}
\label{compint}
Let $I=(x_1^{a_1},\ldots, x_n^{a_n}, x_1^{b_1}\cdots x_n^{b_n})\subseteq S$ be a monomial almost complete intersection. Then $R=S/I$ is of Teter type if and only if there exist $j\not=j'$ such that $2b_j\ge a_j$ and $2b_{j'}\ge a_{j'}$.
\end{Theorem}
\begin{proof}
Without loss of generality we can assume $a_i\ge 2$ for all $i=1,\ldots,n$. We can also assume $b_1b_2\cdots b_r\not=0$ and $b_{r+1}=\cdots=b_n=0$ for some $2\le r\le n$. Let $w=x_1^{b_1}\cdots x_n^{b_n}$ and $\P$ be the divisor poset of $R$.
From \Cref{tracemaci} we know that

$$\tr(\P^*)=(x_1^{a_1-b_1},\ldots, x_n^{a_n-b_n},\frac{w}{x_1^{b_1}},\ldots,\frac{w}{x_n^{b_n}})=(x_1^{a_1-b_1},\ldots, x_n^{a_r-b_r},\frac{w}{x_1^{b_1}},\ldots,\frac{w}{x_r^{b_r}}).$$
Let $v=x_1^{a_1-1}x_2^{a_2-1}\cdots x_n^{a_n-1}$.
Then one can check that $\Soc(\P)=\{\frac{v}{x_1^{a_1-b_1}}, \ldots, \frac{v}{x_r^{a_r-b_r}}\}$.
We consider several cases:
\medskip

(1) There exist $j\not=j'$ such that $2b_j\ge a_j$ and $2b_{j'}\ge a_{j'}$. Without loss of generality, we may assume that $2b_1\ge a_1$ and $2b_2 \ge a_2$. Then $\frac{w}{x_1^{b_1}}$ is divisible by $x_2^{b_2}$ and therefore by $x_2^{a_2-b_2}$. Similarly, $\frac{w}{x_2^{b_2}}$ is divisible by $x_1^{b_1}$, and therefore by $x_1^{a_1-b_1}$. Each $\frac{w}{x_i^{b_i}}$, $3\le i\le r$, is divisible by $x_1^{b_1}x_2^{b_2}$ and thus by $x_1^{a_1-b_1}x_2^{a_2-b_2}$. In other words, $\tr(\P^*)=(x_1^{a_1-b_1},\ldots, x_r^{a_r-b_r})$. We have $\Gen(\tr(\P^*))=\{x_1^{a_1-b_1},\ldots, x_r^{a_r-b_r}\}$ and $\Soc (\tr(\P^*))=\{\frac{v}{x_1^{a_1-b_1}}, \ldots, \frac{v}{x_r^{a_r-b_r}}\}$. By \Cref{symmetric}, this is a symmetric ideal. In fact, the desired isomorphism has multidegree $v$. Therefore, $R$ is of Teter type.

\medskip
(2) There exists at most one $j$ such that $2b_j\ge a_j$. Without loss of generality, we assume $2b_2<a_2, \ldots, 2b_r<a_r$ and no conditions on $b_1$.

(a) $r>2$. Then among monomials $x_1^{a_1-b_1},\ldots, x_r^{a_r-b_r},\frac{w}{x_1^{b_1}}:=w_1$ no pair of monomials divides each other. Indeed, $\supp(w_1)=\{x_2,x_3,\ldots, x_r\}$ which contains at least $2$ variables, which implies that $w_1$ does not divide any other monomial. On the other hand, the $x_i$ exponent of $w_1$ is $b_i<a_i-b_i$ for all $2\le i\le r$ and thus no other monomial divides $w_1$. Therefore,
$|\Gen(\tr(\P^*))|>|\Soc(\tr(\P^*))|$, which by \Cref{symmetric} implies that $\tr(\P^*)$ is not symmetric and thus $R$ is not of Teter type.

(b) $r=2$. In this case $I=(x_1^{a_1},\ldots, x_n^{a_n}, x_1^{b_1}x_2^{b_2})$ with $2b_2<a_2$. Let $e:=\min\{b_1,a_1-b_1\}$. We have
$$\tr(\P^*)=(x_1^{a_1-b_1}, x_2^{a_2-b_2}, x_1^{b_1}, x_2^{b_2})=(x_1^e, x_2^{b_2})=:(t_1,t_2)$$
and
$$\Soc(\tr(\P^*))=\{x_1^{b_1-1}x_2^{a_2-1}, x_1^{a_1-1}x_2^{b_2-1}\}=:\{s_1,s_2\}.$$ The only possible isomorphisms can be given by $\varphi_1: s_1 \mapsto t_1, s_2 \mapsto t_2$ and $\varphi_2: s_1 \mapsto t_2, s_2 \mapsto t_1$. In order for $\varphi_1$ to be an isomorphism, we need to have $(e+b_1-1, a_2-1)=(a_1-1,2b_2-1)$, which is impossible since $2b_2<a_2$. Similarly, in order for $\varphi_2$ to be an isomorphism, we need to have $(b_1-1, a_2+b_2-1)=(e+a_1-1,b_2-1)$, which is also impossible since $a_2\not=0$.
\end{proof}

\begin{Example}
\label{Tschaikovsky}
(1) Let $R = K[x,y]/(x^3,y^3,xy)$. By \Cref{tracemaci} we know that $\tr(\P^*)=(x,y)$. By \Cref{symmetric}, this is not a symmetric ideal and therefore $R$ is not of Teter type in the multigraded sense. However, $R$ is of Teter type in the graded sense (and thus in the general sense) since $\varphi: (x^2)^*\mapsto x, (y^2)^*\mapsto y$ gives a desired graded isomorphism of degree $3$.

(2) Let $R = K[x,y]/(x^3,y^4,xy)$. It is not hard to see that $R=G/(0:\mm_G)$, where $G=K[x,y]/(xy, x^3+y^4)$. Therefore, $R$ is a Teter ring in the local sense. In particular, $R$ is of Teter type in the local sense. However, $R$ is not of Teter type in the graded sense. Indeed, $\Gen(\tr(\P^*))=\{x,y\}$ consists of monomials of the same degree, whereas $\Soc(\tr(\P^*))=\{x^2,y^3\}$ does not.
\end{Example}

Finally, \Cref{Beethoven} below is a special case of \Cref{compint}.  However, its proof is simple and does not depend on \Cref{compint}.

\begin{Theorem}
\label{Beethoven}
Let $S = K[x_1, \ldots, x_s]$ and $w_0 \in S$ a squarefree monomial of degree $> 1$.  Let $\P$ denote the divisor poset of $R = S/(x_1^2, \ldots, x_n^2, w_0)$ and $\tr(\P^*)$ the trace of the divisor poset $\P^*$ of $\omega_{R}$.  Then $\P$ is of Teter type.  Furthermore, $\Gen(\tr(\P^*))$ consists of those variables $x_i$ for which $x_i$ divides $w_0$.
\end{Theorem}

\begin{proof}
Let, say, $w_0 = x_1x_2 \cdots x_{i_0}$ and set $v_0 = x_1x_2 \cdots x_n$.  Then $\Soc(\P^*)$ consists of the monomials $v_0/x_1, v_0/x_2, \ldots, v_0/x_{i_0}$. \Cref{symmetric} guarantees that the poset ideal $\I_0$ with $\Gen(\I_0) = \{x_1, x_2, \ldots, x_{i_0}\}$ is symmetric.  Hence $\I_0 \subset \tr(\P^*)$.  Let $\I$ be a $\tau$-ideal of $\P$ and suppose that $\I \nsubseteq \I_0$.  Let, say, $u = x_{i_0+1}x_{i_0+2} \cdots x_{i_0+d} \in \Gen(\I)$.  Let $\J$ be a companion of $\I$ and $v \in \J$ is accompanied by $u$.  Then there is $1 \leq j_0 \leq i_0$ with $v = v_0/x_{j_0}$.  Since $u$ divides each of the monomials $v_0/x_1, v_0/x_2, \ldots, v_0/x_{i_0}$, it follows that each of the monomials $v_0/x_1u, v_0/x_2u, \ldots, v_0/x_{i_0}u$ divides $v = v_0/x_{j_0}$.  Thus $x_iu/x_{j_0} \in \P$ for each $1 \leq i \leq i_0$.  However, $x_iu/x_{j_0} \in \P$ only for $i = j_0$.  Hence $\I$ cannot be a $\tau$-ideal.  Thus $\I_0 = \tr(\P^*)$, as desired.
\end{proof}

\section{Divisor posets of simplicial complexes}

Let $\Delta$ be a simplicial complex on the vertex set $[n]=\{1, \ldots, n\}$. In other words, $\Delta$ is a collection of subsets of $[n]$ such that $\{i\} \in \Delta$ for $1 \leq i \leq n$ and that, if $F \in \Delta$ and $F' \subseteq F$, then $F' \in \Delta$.  Each element $F \in \Delta$ is called a {\em face} of $\Delta$.  A {\em facet} of $\Delta$ is a face $F$ of $\Delta$ for which $F \subsetneq F'$ for no $F' \in \Delta$.  A face $F$ of $\Delta$ is called {\em free} if there is a unique facet $F'$ of $\Delta$ with $F \subseteq F'$.

Let $S = K[x_1, \ldots, x_n]$ denote the polynomial ring in $n$ variables over a field $K$.  One associates each $F \subseteq [n]$ with the squarefree monomial $u_F = \prod_{i \in F} x_i$.  Let $I_\Delta$ denote the ideal of $S$ which is generated by the monomials $u_F$ with $F \not\in \Delta$.  Let $K\{\Delta\}$ denote the $0$-dimensional monomial $K$-algebra
\[
K\{\Delta\} = S/(I_\Delta, x_1^2, \ldots, x_n^2).
\]
The divisor poset $\P_\Delta$ of $K\{\Delta\}$ is the finite set $\{ u_F : F \in \Delta\}$ with the partial order $\preceq$ defined by $u_F \preceq u_{F'}$ if $F' \subseteq F$.  In particular, $\Soc(\P_\Delta)$ consists of the monomials $u_F$ for which $F$ is a facet of $\Delta$.  The $K$-algebra $K\{\Delta\}$ is Gorenstein if and only if $\Delta$ is the simplex on $[n]$, i.e., $\Delta$ consists of all subsets of $[n]$.  The empty face $\emptyset$ of $\Delta$ is free if and only if $\Delta$ is the simplex on $[n]$.

A simplicial complex $\Delta$ is called {\em flag} if $I_\Delta$ is generated by quadratic monomials.  In other words, $\Delta$ is flag if for any $F \subseteq [n]$ with $F \not\in \Delta$ there is $1 \leq i < j \leq n$ with $\{i, j\} \subseteq F$ such that $\{i, j\} \not\in \Delta$.

\begin{Theorem}
\label{trace_of_squarefree_divisor_poset}
Let $\Delta$ be a simplicial complex on $[n]$
and suppose that $\Delta$ is flag.  Let $\P_{\Delta}$ denote the divisor poset of $K\{\Delta\}$.  Then $\tr(\P_\Delta^*)$ is generated by the  monomials $u_F$ for which $F$ is a free face.
\end{Theorem}

\begin{proof}
Let $F \in \Delta$ be a free face of $\Delta$ and $G$ a unique facet of $\Delta$ with $F \subseteq G$.  In the divisor poset $\P_\Delta$, write $\I$ for the poset ideal with $\Gen(\I) = \{u_F\}$.  Then $\Soc(\I) = \{u_G\}$.  Hence $\I$ is symmetric and then by \Cref{unionsymmetric}, $u_F \in \tr(\P_\Delta^*)$.

Conversely, let $F \in \Delta$ and suppose that $u_F \in \tr(\P_\Delta^*)$. Then $u_F$ is divisible by a minimal generator, say $u_{F'}$.  Therefore by \Cref{unionsymmetric} there is a symmetric ideal $\J$ for which $u_{F'} \in \Gen(\J)$.  Let $G_1, \ldots, G_s$ denote the facets of $\Delta$ containing $F'$.
Since $\J$ is symmetric, by \Cref{symmetric} there is a facet $G_0$ with $u_{G_0} \in \Soc(\J)$ and a face $F_i$ for which $u_{F'}u_{G_0}=u_{F_i}u_{G_i}$ for all $1\leq i\leq s$. This implies that $G_i \setminus F' \subseteq G_0$ and hence for $G = G_1 \cup \cdots \cup G_s$ we have $G \setminus F' \subseteq G_0$.  We show that $G$ is a face of $\Delta$. Since $\Delta$ is flag, it is enough to show that for any $1 \leq i_0 < j_0 \leq n$ with $\{i_0, j_0\}\subseteq G$, we have $\{i_0, j_0\}\in \Delta$.  If $\{i_0, j_0\} \subseteq G \setminus F'\subseteq G_0$ or $\{i_0, j_0\} \subseteq F'$, then we are done. Otherwise, we may assume that $i_0 \in G \setminus F'$ and $j_0 \in F'$. Then $\{i_0, j_0\} \in G_i$ for some $1 \leq i \leq s$.  Hence $\{i_0, j_0\} \in \Delta$.   Therefore $G$ is a face of $\Delta$ which implies $s = 1$ and then $G_1$ is a unique facet containing $F'$. Since $F'\subseteq F $, $F$ is a free face as well.
\end{proof}




\begin{Example}
\label{EXAMPLE_simplicial_complex}
(a)
Let $\Delta$ denote the simplicial complex on $[4]$ whose facets are $\{1,2,3\}$ and $\{3,4\}$.  Then $\tr(\P_\Delta^*) = (x_1, x_2, x_4)$ and by \Cref{symmetric}, $\P_\Delta$ is not of Teter type.  

\medskip

(b)
Let $\Delta$ denote the simplicial complex on $[4]$ whose facets are $\{1,2,3\}$ and $\{1,2,4\}$.  Then $\tr(\P_\Delta^*) = (x_3, x_4)$.  In the divisor poset $\P_\Delta$ the poset ideal $\I$ with $\Gen(\I) = \{x_3, x_4\}$ and $\Soc(\I) = \{x_1x_2x_3, x_1x_2x_4\}$ is symmetric by \Cref{symmetric}.  Thus $\P_\Delta$ is of Teter type.

\medskip

(c)
Let $\Delta$ denote the simplicial complex on $[5]$ whose facets are $$\{1,2,3\}, \, \, \, \{1,2,4\}, \, \, \, \{1,2,5\}.$$  Then $\tr(\P_\Delta^*) = (x_3, x_4, x_5)$ and $\P_\Delta$ is not of Teter type. 

\medskip

(d)
Let $\Delta$ denote the simplicial complex on $[6]$ whose facets are $$\{1,4,5\}, \, \, \, \{2,5,6\}, \, \, \, \{3,4,6\}, \, \, \, \{4,5,6\}.$$  Then $\tr(\P_\Delta^*) = (x_1, x_2, x_3, x_4x_5x_6)$ and $\P_\Delta$ is not of Teter type. 

\medskip

(e)
Let $\Delta$ denote the simplicial complex on $[6]$ whose facets are $$\{1,4\}, \, \, \, \{2,5\}, \, \, \, \{3,6\}, \, \, \, \{4,5,6\}.$$  Then $\tr(\P_\Delta^*) = (x_1, x_2, x_3, x_4x_5, x_5x_6, x_4x_6)$ and $\P_\Delta$ is not of Teter type.  

\medskip

(f)
Let $\Delta$ denote the simplicial complex on $[4]$ whose facets are $\{1,2\}$ and $\{3,4\}$.  Then $\tr(\P_\Delta^*) = (x_1, x_2, x_3, x_4)$.  Thus $K\{\Delta\}$ is nearly Gorenstein.  However, $\P_{\Delta}$ is not of Teter type.

\medskip

(g)
Let $n \geq 3$ and $\Delta_n$ the simplicial complex on $[n]$ whose facets are
\[
\{1, 2\}, \{2, 3\}, \ldots, \{n-1,n\}, \{1,n\}.
\]
If $n > 3$, then $\Delta_n$ is flag and $\tr(\P_\Delta^*) = (x_1x_2, x_2x_3, \ldots, x_{n-1}x_n, x_1x_n)$.  On the other hand, $\Delta_3$ is not flag.  The free faces of $\Delta_3$ are $\{1,2\}, \{2,3\}, \{1,3\}$.
However, $\tr(\P_\Delta^*) = (x_1, x_2, x_3)$.
This example shows that the assumption of being flag in \Cref{trace_of_squarefree_divisor_poset} can not be removed.
Note that $\Delta_n$ is of Teter type if and only if $n\in \{3,4\}$.

\medskip

\end{Example}

\begin{Example}
\label{real_projective_plane}

Let $\Delta$ be the standard triangulation of the real projective plane drawn below.

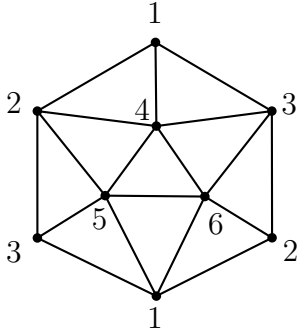
\begin{figure}[hbt]
\begin{center}
\psset{unit=0.78cm}

\begin{pspicture}(5.,5.)(15.,12.)

\psline(10.012710743801657,10.60320961682946)(7.999187077385428,9.431158527422998)
\psline(10.012710743801657,10.60320961682946)(11.996181818181823,9.431158527422998)
\psline(11.996181818181823,9.431158527422998)(12.,7.27)
\psline(7.999187077385428,9.431158527422998)(8.,7.26737)
\psline(12.,7.27)(10.027737039819689,6.2756363636363695)
\psline(8.,7.26737)(10.027737039819689,6.2756363636363695)
\psline(10.027737039819689,9.175711495116461)(9.156211870773857,7.988634109691968)
\psline(10.027737039819689,9.175711495116461)(10.854183320811424,7.973607813673937)
\psline(9.156211870773857,7.988634109691968)(10.854183320811424,7.973607813673937)
\psline(10.027737039819689,9.175711495116461)(10.012710743801657,10.60320961682946)
\psline(10.027737039819689,9.175711495116461)(7.999187077385428,9.431158527422998)
\psline(9.156211870773857,7.988634109691968)(7.999187077385428,9.431158527422998)
\psline(9.156211870773857,7.988634109691968)(8.,7.26737)
\psline(9.156211870773857,7.988634109691968)(10.027737039819689,6.2756363636363695)
\psline(10.854183320811424,7.973607813673937)(10.027737039819689,6.2756363636363695)
\psline(10.854183320811424,7.973607813673937)(12.,7.27)
\psline(10.854183320811424,7.973607813673937)(11.996181818181823,9.431158527422998)
\psline(10.027737039819689,9.175711495116461)(11.996181818181823,9.431158527422998)

\psdot(10.012710743801657,10.60320961682946)
\rput(10,11.10) {$1$}
\psdot(7.999187077385428,9.431158527422998)
\rput(7.6,9.57) {$2$}
\psdot(11.996181818181823,9.431158527422998)
\rput(12.29,9.56) {$3$}
\psdot(8.,7.26737)
\rput(7.6,7.03) {$3$}
\psdot(10.027737039819689,6.2756363636363695)
\rput(10,5.90) {$1$}
\psdot(12.,7.27)
\rput(12.31137, 7.04769) {$2$}
\psdot(10.027737039819689,9.175711495116461)
\rput(9.78409, 9.44769) {$4$}
\psdot(9.156211870773857,7.988634109691968)
\rput(9.05682, 7.59314) {$5$}
\psdot(10.854183320811424,7.973607813673937)
\rput(11.04046, 7.51) {$6$}
\end{pspicture}
\end{center}
\caption{Triangulation of real projective plane}
\label{Triangulation of real projective plane}
\end{figure}

\noindent
Note that $\Delta$ is not flag. Computation in Macaulay2 shows that $\tr(\P_\Delta^*)$ is generated by $x_ix_j$ with $1 \leq i < j \leq 6$. This example once again shows that the assumption of being flag in \Cref{trace_of_squarefree_divisor_poset} can not be removed.

\end{Example}

Let, in general, $P = \{a_1, \ldots, a_n \}$ be a finite poset and $L = \Jc(P)$ be the distributive lattice (\cite[Theorem 9.1.7]{HH}) consisting of all poset ideals of $P$, ordered by inclusion.  A {\em chain} of $L$ is a totally ordered subset of $L$.  A maximal chain of $L$ is a chain $C$ for which $C \subsetneq C'$ for no chain $C'$ of $L$.  The {\em order complex} of $P$ is the simplicial complex $\Delta(P)$ on $[n]$ whose faces are those $F \subseteq [n]$ such that $\{a_i : i \in F\}$ is a chain of $P$.  A {\em linear extension} of $P$ is a permutation $\pi : i_1 i_2 \cdots i_n$ of $[n]$ for which $a_{i_j} < a_{i_k}$ in $P$ implies  $j < k$.  If $\pi : i_1 i_2 \cdots i_n$ is a linear extension of $P$, then
\[
C_\pi : \emptyset \subset \{a_{i_1}\} \subset \{a_{i_1}, a_{i_2}\} \subset \cdots \subset \{a_{i_1}, \cdots, a_{i_n}\} = P
\]
is a maximal chain of $L$.  Furthermore, it can be seen that each maximal chain of $L$ is of the form $C_\pi$ for some linear extension $\pi$ of $P$.

Given a linear extension $\pi : i_1 i_2 \cdots i_n$ of $P$, let $j(1)$ denote the biggest integer for which $a_{i_1} < a_{i_2} < \cdots < a_{i_{j(1)}}$, and let $j(2)$ denote the biggest integer for which $a_{i_{j(1)+1}} < a_{i_{j(1)+2}} < \cdots < a_{i_{j(2)}}$.  Continuing these procedure yields a sequence $1 \leq j(1) < j(2) < \cdots < j(s-1) < j(s) = n$ of integers.  One defines the chain
\[
C_\pi^{\sharp}: \alpha_1 \subset \alpha_2 \subset \cdots \subset \alpha_{s-1},
\]
of $L$, where
\[
\alpha_q = \{ a_{i_1}, a_{i_2}, \ldots, a_{i_{j(q)}}\}.
\]

In particular, if $P$ is a chain, then $j(1) = n$ and $C_\pi^{\sharp} = \emptyset$.

\begin{Example}\label{posetll}
Let $P = \{a_1, a_2, a_3, a_4\}$ be the finite poset and $L = \Jc(P)$ the distributive lattice drawn below.

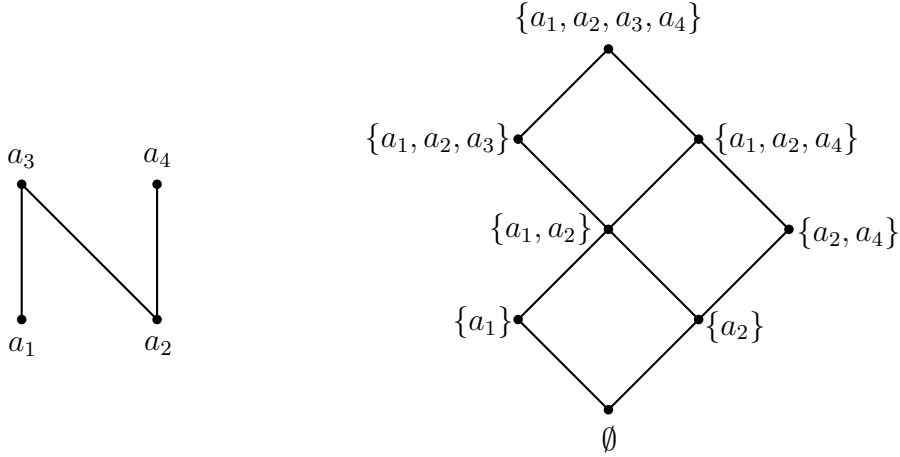
\begin{figure}[hbt]
\begin{center}
\psset{unit=0.6cm}

\begin{pspicture}(2.,3.)(24.,16.)
\psline(5.,8.)(5.,11.)
\psline(8.,11.)(8.,8.)
\psline(5.,11.)(8.,8.)
\psline(18.,6.)(16.,8.)
\psline(18.,6.)(20.,8.)
\psline(16.,8.)(18.,10.)
\psline(20.,8.)(18.,10.)
\psline(20.,8.)(22.,10.)
\psline(18.,10.)(20.,12.)
\psline(22.,10.)(20.,12.)
\psline(16.,12.)(18.,10.)
\psline(18.,14.)(20.,12.)
\psline(18.,14.)(16.,12.)

\psdot(5.,8.)
\psdot(5.,11.)
\psdot(8.,8.)
\psdot(8.,11.)
\psdot(16.,8.)
\psdot(18.,6.)
\psdot(20.,8.)
\psdot(18.,14.)
\psdot(22.,10.)
\psdot(20.,12.)
\psdot(16.,12.)
\psdot(18.,10.)

\rput(5.01937, 7.40509) {$a_1$}
\rput(8.03073, 7.42504) {$a_2$}
\rput(4.99943, 11.57313) {$a_3$}
\rput(8.01078, 11.59308) {$a_4$}

\rput(18.02, 5.35) {$\emptyset$}
\rput(15.22, 7.94) {$\{a_1\}$}
\rput(20.8, 7.84) {$\{a_2\}$}
\rput(16.5, 10.) {$\{a_1,a_2\}$}
\rput(23.32, 9.91) {$\{a_2,a_4\}$}
\rput(21.94, 12) {$\{a_1,a_2,a_4\}$}
\rput(14.22, 12) {$\{a_1,a_2,a_3\}$}
\rput(18., 14.68421) {$\{a_1,a_2,a_3,a_4\}$}
\end{pspicture}
\end{center}
\caption{The poset $P$ and the lattice $\Jc(P)$}
\label{Example 6.6}
\end{figure}

\noindent
Then the linear extensions of $P$ are
\[
\pi_1:1234, \, \, \, \, \pi_2:1243, \, \, \, \, \pi_3:2134, \, \, \, \, \pi_4:2143, \, \, \, \, \pi_5:2413
\]
and
\begin{eqnarray*}
C_{\pi_1}^{\sharp} & : & \{a_1\} \subset \{a_1,a_2,a_3\}, \\
C_{\pi_2}^{\sharp} & : & \{a_1\} \subset \{a_1,a_2,a_4\}, \\
C_{\pi_3}^{\sharp} & : & \{a_2\} \subset \{a_1,a_2,a_3\}, \\
C_{\pi_4}^{\sharp} & : & \{a_2\} \subset \{a_1,a_2\} \subset \{a_1,a_2,a_4\}, \\
C_{\pi_5}^{\sharp} & : & \{a_2,a_4\}.
\end{eqnarray*}
\end{Example}

We see that in \Cref{posetll}, $C_\pi$ is a unique maximal chain which contains $C_\pi^{\sharp}$. One can easily check that this is true in general.  Conversely,

\begin{Lemma}
\label{linear_extension}
Let $C_\pi$ be a maximal chain of $L$ and $C$ be a chain of $L$ for which $C_\pi$ is a unique maximal chain which contains $C$.  Then $C_\pi^{\sharp} \subset C$.
\end{Lemma}

\begin{proof}
Following the notation $\pi, j(q)$ and $C_\pi^{\sharp}$ as above, let
$
C: \beta_1 \subset \beta_2 \subset \cdots \subset \beta_r.
$
We proceed by induction on $r$.  Clearly, $\beta_1 \subset \alpha_1$.  If $\beta_1 = \alpha_1$, then considering the interval
\[
[\alpha_1, P] = \{\alpha \in \Jc(P) : \alpha_1 \subset \alpha \},
\]
which is a distributive lattice, shows that $C_\pi^{\sharp} \setminus \alpha_1 \subset C \setminus \beta_1$.  Thus $C_\pi^{\sharp} \subset C$.  Suppose that $\beta_1 \subsetneq \alpha_1$.  If $\beta_2 \subset \alpha_1$, then $C_\pi$ is a unique maximal chain containing $C \setminus \beta_1$.  Thus $C_\pi^{\sharp} \subset C \setminus \beta_1$.  In particular, $C_\pi^{\sharp} \subset C$.  Let $\beta_1 \subsetneq \alpha_1 \subsetneq \beta_2$.  It then follows that $\beta_2 \setminus \beta_1$ is a chain of $P$.  Since $\alpha_1$ is a chain of $P$, it follows that
\[
a_{i_1} < a_{i_2} < \cdots < a_{i_{j(1)}} < a_{i_{j(1)+1}}.
\]
This contradicts the definition of $j(1)$.  Thus $\beta_1 \subsetneq \alpha_1 \subsetneq \beta_2$ cannot occur.
\end{proof}

\begin{Theorem}
\label{distributive_trace}
Let $L = \Jc(P)$ and let $\Delta(L)$ be its order complex. Then $\Gen(\tr(\P_{\Delta(L)}^*))$ consists of the monomials $u_{C_\pi^{\sharp}}$, where $\pi$ is a linear extension of $P$.  
\end{Theorem}

\begin{proof}
Since $\Delta(L)$ is flag, the desired result follows from \Cref{trace_of_squarefree_divisor_poset} together with \Cref{linear_extension}.
\end{proof}

\begin{Corollary}
\label{distributive_of_teter_type}
Let $L = \Jc(P)$ and let $\Delta(L)$ be its order complex.  Then $\P_{\Delta(L)}$ is of Teter type in the category of 0-dimensional local $K$-algebras.
\end{Corollary}

\begin{proof}
Let $\I_\pi$ denote the interval $[u_{C_\pi^\sharp}, u_{C_\pi}]$ of $\P_{\Delta(L)}$, where $\pi$ is a linear extension of $P$.  It follows  from Corollary $2.3$ that $\I_\pi$ is a (multigraded) symmetric poset ideal.  Furthermore, if $\pi \neq \pi'$ then $\I_\pi \cap \I_{\pi'} = \emptyset$.  It then follows that  $\tr(\P_{\Delta(L)}^*) = \sqcup_\pi \I_\pi$, where $\pi$ ranges over the linear extensions of $P$. This implies that $\tr(\P_{\Delta(L)}^*)$ is a symmetric poset ideal in the local sense.  Hence $K\{\Delta(L)\}$ is of Teter type in the local sence.
\end{proof}

We now turn to the study of the independence complex of a finite simple graph.  Let $G$ be a finite simple graph on the vertex set $[n]$ and $E(G)$ be the set of edges of $G$.  We say that a subset $F \subset [n]$ is {\em independent} in $G$ if, for each $i \in F$ and $j \in F$ with $i \neq j$, one has $\{i, j \} \not\in E(G)$.  Thus in particular $\{i\}$ is independent for each $i \in [n]$.  Let $\Delta(G)$ denote the simplicial complex on $[n]$ whose faces are the independent subsets of $G$.  The simplicial complex $\Delta(G)$ is called the {\em independence complex} of $G$.  The independence complex $\Delta(G)$ is flag.  In fact,
\[
I_{\Delta(G)} = (x_i x_j : \{i,j\} \in E(G)).
\]
Conversely, given a flag complex $\Delta$ on $[n]$, there is a unique finite simple graph $G$ on $[n]$ with $\Delta = \Delta(G)$.  It would be of interest to describe the trace of $\P^*_{\Delta(G)}$ in terms of the combinatorics of $G$.

A {\em path graph} of length $n - 1$ is the finite simple graph $P_{n}$ whose edges are
\[
\{1,2\}, \, \, \{2,3\}, \, \, \ldots, \{n-1,n\}.
\]

\noindent
Our job is to find an explicit combinatorial description of $\Gen(\tr(\P^*_{\Delta(P_{n})}))$.

Let $(a_1, a_2, \ldots, a_s)$ be a sequence of integers with $1 \leq a_1 < a_2 < \cdots < a_s \leq n$.  We say that $(a_1, a_2, \ldots, a_s)$ is {\em permissible} if
\[
a_{i+1} - a_i \in \{2, 3\}, \, \, \, \, \, 0 \leq i \leq s,
\]
with setting $a_{0} = -1$ and $a_{s+1} = n + 2$.  Furthermore, we say that $(a_1, a_2, \ldots, a_s)$ is {\em $\tau$-permissible}  if $(a_{i-1}, a_i, a_{i+1}) = (a_i-2, a_i, a_i+2)$ for no $1 \leq i \leq s$ and if
\[
a_{i+1} - a_i \in \{2,3,4\}, \, \, \, \, \, 1 \leq i \leq s.
\]

\begin{Example}
\label{soc_gen_P_6}
Let $n = 7$. Then the permissible sequences are
\[
(1,3,5,7), (1,3,6), (1,4,6), (1,4,7), (2,4,6), (2,4,7), (2,5,7).
\]
The $\tau$-permissible sequences are
\[
(1,5), (3,5), (3,7), (3,6), (1,4,6), (1,4,7), (2,6), (2,4,7), (2,5).
\]
\end{Example}

\begin{Theorem}
\label{path_odd_even}
Let $n \geq 2$ and $\Delta(P_{n})$ the independence complex of $P_{n}$.  Let $F = \{a_1, a_2, \ldots, a_s\}$, where $1 \leq a_1 < a_2 < \cdots < a_s \leq n$, be a subset of $[n]$.  Then
\begin{itemize}
\item[(a)]
$u_F$ belongs to $\Soc(\P_{\Delta(P_{n})})$ if and only if $(a_1, a_2, \ldots, a_s)$ is permissible;
\item[(b)]
$u_F$ belongs to $\Gen(\tr(\P^*_{\Delta(P_{n})}))$ if and only if $(a_1, a_2, \ldots, a_s)$ is $\tau$-permissible.
\end{itemize}
\end{Theorem}

\begin{proof}
(a)
It follows that $F = \{a_1, \ldots, a_s\}$ is a maximal independent subset of $P_{n}$ if and only if $(a_1, \ldots, a_s)$ is permissible.  Since $u_F$ belongs to $\Soc(\P_{\Delta(P_{n})})$ if and only if $F$ is a facet of $\Delta(P_{n})$, the desired result follows.

\medskip

(b)
Let $u_F\in \Gen(\tr(\P^*_{\Delta(P_{n})}))$. Then by \Cref{trace_of_squarefree_divisor_poset}, $F$ is a minimal free face of $\Delta$. On the contrary assume that $(a_1, a_2, \ldots, a_s)$ is not $\tau$-permissible.
Then there exists an integer $i$ such that either $a_{i+2}-a_i=a_i-a_{i-2}=2$ or $a_{i+1}-a_i\geq 5$. In the former case $F\setminus \{a_i\}$ is a free face as well which contradicts the minimality of $F$. In the latter case $F$ is not a free face. Indeed if by contradiction $G$ is the unique facet containing $F$, then $F_1=F\cup \{a_i+2\}$
and $F_2=F\cup \{a_i+3\}$ are also contained in $G$ which is impossible. Hence $F$ is not a free face and we get a contradiction.

Conversely, let $(a_1, \ldots, a_s)$ be $\tau$-permissible.  Let $W$ denote the set of integers $i$ for which $a_{i+1} - a_i = 4$, where $0 \leq i \leq s$.  It then follows that
\[
G = \{a_1, \ldots, a_s\} \cup \{ a_i + 2 : i \in W \}
\]
is a maximal independent subset of $P_{n}$ and that $F = \{a_1, \ldots, a_s\}$ is a free face of $\Delta(P_{n})$ with $F \subseteq G$.
From the definition of a $\tau$-permissible sequence one can see that $F$ is minimal among the free faces contained in $G$.  Hence the desired result follows from \Cref{trace_of_squarefree_divisor_poset}.
\end{proof}

\begin{Example}
\label{AGAIN_soc_gen_P_6}
Let $n = 9$.  Then $(1,3,5,7,9)$ is permissible.  Following the proof of \Cref{path_odd_even} with $G = \{1,3,5,7,9\}$ and removing $5, 1, 9$ in this order, one has $F_0 = \{1,3,7,9\}, F_1 = \{3,7,9\}, F_2 = \{3,7\}$ and $(3,7)$ is $\tau$-permissible.  On the other hand, removing $3, 7$ in this order, one has $F_0 = \{1,5,7,9\}, F_1 = \{1,5,9\}$ and $(1,5,9)$ is $\tau$-permissible.  Furthermore, $(1,5,7)$ and $(3,5,9)$ are also $\tau$-permissible, each of which is a subsequence of $(1,3,5,7,9)$.
\end{Example}

A {\em cycle of length $n$} is the finite simple graph $C_n$ on $[n]$ whose edges are
\[
\{1,2\}, \, \, \{2,3\}, \, \, \ldots, \{n-1,n\}, \, \, \, \{1, n\}.
\]
Let $i \in [n]$ and $j \in [n]$.  The {\em distance} between $i$ and $j$ in $C_n$ is
\[
{\rm dist}(i,j) = \min\{|j - i|, n - |j - i|\}.
\]
Let $F \subset [n]$.  We say that $i \in F$ and $j \in F$ are {\em adjacent} in $F$ if there is a sequence of edges
\[
(i_0, i_1), (i_1, i_2), (i_2, i_3), \ldots, (i_{s-2}, i_{s-1}), (i_{s-1}, i_s)
\]
for which $i_0 = i, \, i_s = j$ and none of $i_1, \ldots, i_{s-1}$ belongs to $F$.

Let $n \geq 3$ and $F \subset [n]$ with $|F| \geq 3$.  We say that $F$ is {\em permissible} if, for $i \in F$ and $j \in F$ which are adjacent in $F$, one has ${\rm dist}(i,j) \in \{2,3\}$.  Furthermore, we say that $F$ is {\em $\tau$-permissible} if the following conditions are satisfied:
\begin{itemize}
\item[(i)]
If $i \in F$ and $j \in F$ are adjacent in $F$, then ${\rm dist}(i,j) \in \{2,3,4\}$;
\item[(ii)]
There is no $i_0 \in F$ for which ${\rm dist}(i_0,i) = 2$ if $i \in F$ and $i_0$ are adjacent in $F$.
\end{itemize}
For example, in $C_9$, $\{1, 3, 5, 7\}$ is permissible and $\{1, 5, 7\}$ is $\tau$-permissible.  In $C_8$, $\{1, 4, 6\}$ is permissible and $\{1, 5\}$ is $\tau$-permissible.  In $C_5$, $\{1, 3\}$ is permissible as well as $\tau$-permissible.  In $C_4$, $\{1,3\}, \, \{2,4\}$ are permissible and $\{1\}, \{2\}, \{3\}, \{4\}$ are $\tau$-permissible.

We now come to an explicit combinatorial description of ${\rm Gen}(\tr(\P^*_{\Delta(C_n)}))$.  A proof of Theorem \ref{cycle_of_length_n} is  similar to that of Theorem \ref{path_odd_even} and is omitted.

\begin{Theorem}
\label{cycle_of_length_n}
Let $n \geq 3$ and $\Delta(C_n)$ the independence complex of $C_n$.  Let $F \subset [n]$.  Then
\begin{itemize}
\item[(a)]
$u_F$ belongs to $\Soc(\P_{\Delta(C_n)})$ if and only if one of the following conditions are satisfied:

{\rm (i)} $|F| \geq 3$ and $F$ is permissible;

{\rm (ii)} $|F| = 2$, say, $F = \{i, j\}$, and $${\rm dist}(i,j) \in \{2, 3\}, \, \, \, n - {\rm dist}(i,j) \in \{2, 3\};$$

{\rm (iii)} $|F| = 1$ and $n = 3$.

\item[(b)]
$u_F$ belongs to ${\rm Gen}(\tr(\P^*_{\Delta(C_n)}))$ if and only if the following conditions are satisfied:

{\rm (i)} $|F| \geq 3$ and $F$ is $\tau$-permissible;

{\rm (ii)} $|F| = 2$, say, $F = \{i, j\}$, and $${\rm dist}(i,j) \in \{2, 3, 4\}, \, \, \, n - {\rm dist}(i,j) \in \{2, 3, 4\};$$

{\rm (iii)} $|F| = 1$ and $n = 3$.
\end{itemize}
\end{Theorem}

\end{document}